\documentclass{amsart}

\usepackage{bbm}
\usepackage{amssymb}
\usepackage{hyperref}

\newcommand*{\mailto}[1]{\href{mailto:#1}{\nolinkurl{#1}}}
\newcommand{\arxiv}[1]{\href{http://arxiv.org/abs/#1}{arXiv: #1}}
\newcommand{\doi}[1]{\href{http://dx.doi.org/#1}{doi: #1}}

\newtheorem{theorem}{Theorem}[section]
\newtheorem{definition}[theorem]{Definition}
\newtheorem{lemma}[theorem]{Lemma}
\newtheorem{proposition}[theorem]{Proposition}
\newtheorem{corollary}[theorem]{Corollary}

\newtheorem{remark}[theorem]{Remark}

\newcommand{\R}{{\mathbb R}}
\newcommand{\N}{{\mathbb N}}

\newcommand{\C}{{\mathbb C}}

\newcommand{\spr}[2]{\langle #1 , #2 \rangle}
\newcommand{\indik}{\mathbbm{1}}

\newcommand{\f}{\mathbf{f}}
\newcommand{\g}{\mathbf{g}}

\newcommand{\supp}{\mathrm{supp}}

\newcommand{\D}{\mathcal{D}}
\newcommand{\F}{\mathcal{F}}
\newcommand{\G}{\mathcal{G}}
\newcommand{\Qr}{\mathsf{q}}
\newcommand{\Nr}{\mathsf{n}}
\newcommand{\Hr}{\mathsf{h}}

\newcommand{\Wr}{\mathsf{w}}
\newcommand{\T}{\mathrm{T}}

\newcommand{\I}{\mathrm{i}}

\newcommand{\tr}{\mathrm{tr}}
\newcommand{\im}{\mathrm{Im}}
\newcommand{\re}{\mathrm{Re}}

\newcommand{\be}{\begin{equation}}
\newcommand{\ee}{\end{equation}}
\newcommand{\ba}{\begin{array}}
\newcommand{\ea}{\end{array}}
\newcommand{\loc}{\mathrm{loc}}
\newcommand{\cc}{\mathrm{c}}

\newcommand{\qd}{{[1]}}

\newcommand{\dom}[1]{\mathrm{dom}(#1)}
\newcommand{\ran}[1]{\mathrm{ran}(#1)}
\renewcommand{\ker}[1]{\mathrm{ker}(#1)}
\newcommand{\mul}[1]{\mathrm{mul}(#1)}

\newcommand{\ledot}{\,\cdot\,}
\newcommand{\redot}{\cdot\,}
\newcommand{\NL}{(0-)}
\newcommand{\NLz}{(z,0-)}

\newcommand{\Hasto}{H^\prime_\ast[0,L)}
\newcommand{\cH}{\mathcal{H}}
\newcommand{\cK}{\mathcal{K}}
\newcommand{\Lmu}{L^2(\R;\mu)}
\newcommand{\SMP}{\Xi}
\newcommand{\Strings}{\mathcal{S}}
\newcommand{\Nevan}{\mathcal{N}}
\newcommand{\dip}{\upsilon}

\renewcommand{\labelenumi}{(\roman{enumi})}
\renewcommand{\theenumi}{\roman{enumi}}
\numberwithin{equation}{section}

\begin{document}

\title[The inverse spectral problem]{The inverse spectral problem for indefinite strings}

\author[J.\ Eckhardt]{Jonathan Eckhardt}
\address{School of Computer Science \& Informatics\\ Cardiff University\\ Queen's Buildings \\ 5 The Parade\\ Roath \\ Cardiff CF24 3AA\\ Wales \\ UK}
\email{\mailto{EckhardtJ@cardiff.ac.uk}}

\author[A.\ Kostenko]{Aleksey Kostenko}
\address{Faculty of Mathematics\\ University of Vienna \\ Oskar-Morgenstern-Platz 1\\ 1090 Wien\\ Austria}
\email{\mailto{Oleksiy.Kostenko@univie.ac.at}; \mailto{duzer80@gmail.com}}

\thanks{\href{http://dx.doi.org/10.1007/s00222-015-0629-1}{Invent.\ Math.\ {\bf 204} (2016), no.~3, 939--977}}
\thanks{{\it Research supported by the Austrian Science Fund (FWF) under Grants No.\ J3455 and P26060.}}


\begin{abstract}
 Motivated by the study of certain nonlinear wave equations (in particular, the Camassa--Holm equation), we introduce a new class of generalized indefinite strings associated with differential equations of the form 
 \begin{align*}
 - u'' = z\, u\, \omega + z^2 u\, \dip 
 \end{align*}
 on an interval $[0,L)$, where $\omega$ is a real-valued distribution in $H^{-1}_{\loc}[0,L)$, $\dip$ is a non-negative Borel measure on $[0,L)$ and $z$ is a complex spectral parameter.
 Apart from developing basic spectral theory for these kinds of spectral problems, our main result is an indefinite analogue of M.\ G.\ Krein's celebrated solution of the inverse spectral problem for inhomogeneous vibrating strings.
\end{abstract}

\maketitle

\section{Introduction}

 A classical object in spectral theory is the differential equation 
 \begin{align}\label{eqnClaS}
  -u'' = z\, u\,\omega  
 \end{align}
 on an interval $[0,L)$, where $\omega$ is a non-negative Borel measure on $[0,L)$ and $z$ is a complex spectral parameter. 
 The relevance of this particular spectral problem initially stems from the fact that it arises after applying the separation of variables method to the wave equation describing small transversal oscillations of a vibrating string of length $L$ and with mass distribution given by $\omega$.  
 Spectral theory for these kinds of problems was developed by M.\ G.\ Krein in the early 1950's \cite{kr52, kakr74} and subsequently applied to study interpolation and filtration problems for stationary stochastic processes; see \cite{dymc76}. 
 Apart from this, it is also of use in connection with the description of one-dimensional Markov birth and death processes \cite{fe54, ka75}. 

 Probably the most prominent object in spectral theory for~\eqref{eqnClaS} is the so-called Weyl--Titchmarsh function, which encodes all the spectral information. 
 A remarkable and well-known result of M.\ G.\ Krein identifies the totality of all possible Weyl--Titchmarsh functions with the class of so-called Stieltjes functions (we refer to \cite{dymc76, kakr74, ka94, kowa82} for several surveys).  
 In other words, he was able to solve the inverse spectral problem for~\eqref{eqnClaS}. 
 The present article is concerned with further questions in this direction which are not far to seek:
 {\em What happens if $\omega$ is allowed to be a real-valued Borel measure on $[0,L)$ instead of just non-negative? 
 Is there an equally concise analogue of M.\ G.\ Krein's solution of the inverse spectral problem?} 
 Although there are various results about spectral theory for~\eqref{eqnClaS} when $\omega$ is allowed to be indefinite (we only mention \cite{atmi87, bft, be89, fl96, fl14, ko13, ze05} and the references therein), there is still no satisfactory answer to these questions.
 A first guess could suggest that instead of the class of Stieltjes functions (which are, roughly speaking, determined by not having singularities on the negative real axis) one obtains the entire class of Herglotz--Nevanlinna functions (which may have singularities on the whole real axis).  
 However, this is not the case as it turned out that the class of spectral problems~\eqref{eqnClaS} with real-valued Borel measures $\omega$ is too narrow in this respect, even for simple cases. 
 In fact, not even all rational Herglotz--Nevanlinna functions arise as the Weyl--Titchmarsh function of such a spectral problem. 
 The deeper reason for why this fails in the real-valued case in some sense lies in the fact that the class of real-valued Borel measures is not closed with respect to a particular topology, whereas the class of non-negative Borel measures is; cf.\ Proposition~\ref{propSMPcont}. 
 Altogether, it does not seem very likely that there is a simple and concise description of the class of Weyl--Titchmarsh functions that arise from the spectral problem~\eqref{eqnClaS} with real-valued Borel measures $\omega$. 
 
 One way to overcome this problem by means of extending the class of spectral problems was suggested by M.\ G.\ Krein and H.\ Langer \cite{krla79, la76}, who considered the modified differential equation  
 \begin{align}\label{eqnSPla}
  - u'' = z\, u\, \omega + z^2 u\, \dip
 \end{align}
 on an interval $[0,L)$, where $\omega$ is a real-valued Borel measure on $[0,L)$ and $\dip$ is a non-negative Borel measure on $[0,L)$. 
 In particular, they showed in \cite{krla79} that indeed every rational Herglotz--Nevanlinna function arises as the Weyl--Titchmarsh function of such a spectral problem.
 However, the totality of all Weyl--Titchmarsh functions which are obtained in this way is still a proper subset of the class of Herglotz--Nevanlinna functions. 
 In fact, a further generalization of~\eqref{eqnSPla} yielding all Herglotz--Nevanlinna functions which only admit finitely many singularities on the negative real axis was later proposed by H.\ Langer and H.\ Winkler \cite{lawi98}. 
 
 It is the purpose of the present article to show how to extend the class of spectral problems such that the corresponding Weyl--Titchmarsh functions coincide exactly with the class of all Herglotz--Nevanlinna functions. 
 To this end, we will consider the differential equation~\eqref{eqnSPla} on an interval $[0,L)$ with $\omega$ being a real-valued distribution in $H^{-1}_{\loc}[0,L)$ and $\dip$ a non-negative Borel measure on $[0,L)$. 
 Of course, it is not obvious at all how this differential equation has to be understood in this case (also note that it is posed on a half-open interval).  
 For this reason, we will first discuss its precise notion in Section~\ref{secBDE}, as well as basic existence and uniqueness results of solutions. 
 In the subsequent section, we will then introduce an associated linear relation in a suitable Hilbert space, which turns out to be self-adjoint.
 Note that we can not use the usual $L^2([0,L);\omega)$ for our class of coefficients as in \cite{la76} for example. 
 Moreover, although~\eqref{eqnSPla} looks quite similar to a Schr\"odinger equation with an energy-dependent potential (see \cite{hrpr12, jaje72, kau75, sasz96} and the references therein), the approaches employed there do not work in our case due to the low regularity assumptions on our coefficients.  
 For this reason, we had to come up with a new way of treating~\eqref{eqnSPla} which is a modification of the left-definite theory for~\eqref{eqnClaS}. 
 In Section~\ref{secWT} we will introduce the principal Weyl--Titchmarsh function associated with our spectral problem, derive its basic properties and establish its connection to the linear relation. 
 The following section is then devoted to the main result of this article, the aforementioned solution of the corresponding inverse spectral problem, supplemented by continuity properties of the correspondence. 
 In the concluding Section~\ref{secEssPos}, we will show how our solution of the inverse spectral problem fits in with the classical result of M.\ G.\ Krein by deducing it from our solution. 
 Since the proof of our main result relies on L.\ de Branges' solution of the inverse spectral problem for canonical systems, we collect the necessary basic facts in Appendix~\ref{sec:cansys} in order to make the exposition self-contained. 
 
 Although the solution of the inverse spectral problem for~\eqref{eqnSPla} may seem to be of purely academic interest at first, it recently turned out to be important in connection with the integration of certain nonlinear wave equations (namely the Camassa--Holm equation \cite{caho93, besasz00} and its two-component generalization \cite{clz06, hoiv12} as well as the Hunter--Saxton equation \cite{husa91} and the Dym equation \cite{kr75}) by means of the inverse scattering transform.  
 In order to incorporate one of the most intriguing features (that is, finite time blow-up which resembles wave breaking \cite{coes98, mc03, mc04}) of these equations, it is of essential importance to be able to allow $\omega$ to be real-valued.  
  This indefiniteness causes serious problems when dealing with these nonlinear wave equations (noticeable from the discussions in \cite{comc99}, \cite{kau06} or \cite{le05} for example), which is why a lot of the literature actually restricts to the case of strictly positive $\omega$. 
  Most of these complications can already be observed for the prototypical and explicitly solvable example of a peakon--antipeakon collision for the Camassa--Holm equation, described in detail in \cite{besasz01, coes98, ConservMP, wa06}. 
 All this is reflected by the fact that also inverse spectral theory for~\eqref{eqnClaS} is much more intricate for real-valued $\omega$ and up to now only insufficient partial results are available in this case \cite{besasz00, besasz07, bel87, bebrwe09, LeftDefiniteSL, CHPencil, IsospecCH}.
 The results of the present article can be used to considerably generalize the best results known so far obtained previously in~\cite{bebrwe09, LeftDefiniteSL, CHPencil, IsospecCH} to rather irregular coefficients $\omega$ which could not be treated before.  
 In view of applications to nonlinear wave equations, the regularity of coefficients assumed in this article is exactly what is desired (see \cite{grhora12} for the two-component Camassa--Holm system). 
 Also note that the references~\cite{bebrwe09, LeftDefiniteSL, CHPencil, IsospecCH} only treat the uniqueness part of the inverse spectral problem and except for \cite{CHPencil}, they do not even include the additional Borel measure $\dip$.
 However, it is exactly this second term in~\eqref{eqnSPla} and the low regularity assumption on $\omega$ which allow us to prove existence for the inverse spectral problem as well.   


\section{Preliminaries} 

 In order to fix our notation, we will first introduce several spaces of functions and distributions. 
 For every fixed $L\in(0,\infty]$, we denote with $H^1_{\loc}[0,L)$, $H^1[0,L)$ and $H^1_{\cc}[0,L)$ the usual Sobolev spaces. 
 To be precise, this means  
\begin{align}
H^1_{\loc}[0,L) & =  \lbrace f\in AC_{\loc}[0,L) \,|\, f'\in L^2_{\loc}[0,L) \rbrace, \\
 H^1[0,L) & = \lbrace f\in H^1_{\loc}[0,L) \,|\, f,\, f'\in L^2[0,L) \rbrace, \\ 
 H^1_{\cc}[0,L) & = \lbrace f\in H^1[0,L) \,|\, \supp(f) \text{ compact in } [0,L) \rbrace.
\end{align}
Furthermore, we introduce the space of distributions $H^{-1}_{\loc}[0,L)$ as the topological dual of $H^1_{\cc}[0,L)$. 
One notes that the mapping $\Qr\mapsto\chi$, defined by
 \begin{align}
    \chi(h) = - \int_0^L \Qr(x)h'(x)dx, \quad h\in H^1_{\cc}[0,L),
 \end{align} 
 establishes a one-to-one correspondence between $L^2_{\loc}[0,L)$ and $H^{-1}_{\loc}[0,L)$. 
The unique function $\Qr\in L^2_{\loc}[0,L)$ corresponding to some $\chi\in H^{-1}_{\loc}[0,L)$ in this way will be referred to as {\em the normalized anti-derivative} of $\chi$.
 We say that the distribution $\chi\in H^{-1}_{\loc}[0,L)$ is real-valued if $\Qr$ is real-valued almost everywhere on $[0,L)$. 

 Let us mention that particular kinds of distributions in $H^{-1}_{\loc}[0,L)$ arise from Borel measures on $[0,L)$.
 In fact, if $\chi$ is a complex-valued Borel measure on $[0,L)$, then we will identify it with the distribution in $H^{-1}_{\loc}[0,L)$ given by  
 \begin{align}
  h \mapsto \int_{[0,L)} h\,d\chi. 
 \end{align}
The normalized anti-derivative $\Qr$ of $\chi$ is simply given by the distribution function
 \begin{align}
 \Qr(x)=\int_{[0,x)}d\chi
 \end{align}
 for almost all $x\in [0,L)$, as an integration by parts (use, for example, \cite[Exercise~5.8.112]{bo07}, \cite[Theorem~21.67]{hest65}) shows.  

 In order to obtain a self-adjoint realization of the differential equation~\eqref{eqnSPla} in a suitable Hilbert space, we finally also introduce the less common function space  
\begin{align}
 H'[0,L) & = \begin{cases} \lbrace f\in H^1_{\loc}[0,L) \,|\, f'\in L^2[0,L),~ \lim_{x\rightarrow L} f(x) = 0 \rbrace, & L<\infty, \\ \lbrace f\in H^1_{\loc}[0,L) \,|\, f'\in L^2[0,L) \rbrace, & L=\infty, \end{cases} 
\end{align}
 as well as the linear subspace   
\begin{align}
 \Hasto & = \lbrace f\in H'[0,L) \,|\, f(0) = 0 \rbrace, 
\end{align}
 which turns into a Hilbert space when equipped with the scalar product 
 \begin{align}\label{eq:normti}
 \spr{f}{g}_{\Hasto} = \int_0^L f'(x) g'(x)^\ast dx, \quad f,\, g\in\Hasto.
 \end{align}
The space $\Hasto$ can be viewed as a completion of the space of all smooth functions which have compact support in $(0,L)$ with respect to the norm induced by~\eqref{eq:normti}. 
In particular, if $L$ is finite, then the space $\Hasto$ coincides algebraically and topologically with the usual Sobolev space $H^1_0[0,L)$. 
Finally, let us mention the following simple growth estimate for functions in $\Hasto$
\begin{align}\label{eqnFuncHastGrow}
 |f(x)|^2 \leq x\left(1-\frac{x}{L}\right)\, \|f\|^2_{\Hasto}, \quad x\in[0,L),~f\in\Hasto, 
\end{align}
 which is obtained from an application of the Cauchy--Schwarz inequality.
 Here and henceforth we employ the convention that whenever an $L$ appears in a denominator, the corresponding fraction has to be interpreted as zero if $L$ is infinite.

\section{The basic differential equation}\label{secBDE}

 Throughout this article, we fix some $L\in(0,\infty]$, let $\omega\in H^{-1}_{\loc}[0,L)$ be a real-valued distribution on $[0,L)$ and $\dip$ be a non-negative Borel measure on $[0,L)$.  
 As a first step, we will discuss the meaning of the inhomogeneous differential equation 
 \begin{align}\label{eqnDEinho}
  -f''  = z\, \omega f + z^2 \dip f + \chi, 
 \end{align}
 where $\chi\in H^{-1}_{\loc}[0,L)$ and $z$ is a complex spectral parameter. 
 Of course, this differential equation has to be understood in a distributional sense (cf.\ \cite{gewe14, ss03}).   
  
  \begin{definition}\label{defSolution}
  A solution of~\eqref{eqnDEinho} is a function $f\in H^1_{\loc}[0,L)$ such that 
 \begin{align}
  \Delta_f h(0) + \int_{0}^L f'(x) h'(x) dx = z\, \omega(fh) + z^2 \dip(fh) +  \chi(h), \quad h\in H^1_{\cc}[0,L),
 \end{align}
 for some constant $\Delta_f\in\C$. 
 In this case, the constant $\Delta_f$ is uniquely determined and will henceforth always be denoted with $f'\NL$ for apparent reasons. 
 \end{definition}
 
 Of course, there are also several other ways of introducing the same notion of solutions. 
 Upon choosing particular test functions $h_x\in H^1_{\cc}[0,L)$ given by 
 \begin{align}
  h_x(t) = \begin{cases} x-t, & t\in[0,x), \\ 0, & t\in[x,L), \end{cases}
 \end{align}
 for every $x\in[0,L)$, one observes that a function $f\in H^1_{\loc}[0,L)$ is a solution of~\eqref{eqnDEinho} if and only if one has 
 \begin{align}\label{eqnDEinhoInt}
   f(x) & = f(0) + \Delta_f x - z\,\omega(fh_x) - z^2 \dip(fh_x) - \chi(h_x), \quad x\in[0,L), 
 \end{align}
 for some $\Delta_f\in\C$ (which is the same constant as in the definition above). 
  Note that this formulation simply reduces to the usual integral equation (as used in, for example, \cite[\S1]{kakr74}, \cite[Section~1]{la76}) if $\omega$ and $\chi$ are Borel measures.
 Alternatively, upon denoting with $\Nr_z$ and $\Qr$ the respective normalized anti-derivatives of the distributions $z\,\omega+z^2\dip$ and $\chi$, the equation~\eqref{eqnDEinhoInt} takes the form  
 \begin{align}\label{eqnDEinhoIntNAD}
   f(x) & = f(0) + \Delta_f x -  \int_0^x \Nr_z(t) (f(t) - f'(t)(x-t))dt - \int_0^x \Qr(t)dt, \quad x\in[0,L).
 \end{align}
 From this, one readily verifies that a function $f\in H^1_{\loc}[0,L)$ is a solution of~\eqref{eqnDEinho} if and only if one has 
 \begin{align}\label{eqnQD}
  f'(x) + \Nr_z(x)f(x) + \Qr(x) = \Delta_f + \int_0^x \Nr_z(t)f'(t) dt
 \end{align}
 for some $\Delta_f\in\C$  (which is again the same constant as before) and almost all $x\in[0,L)$. 
 In fact, to this end one simply needs to differentiate~\eqref{eqnDEinhoIntNAD}, respectively integrate~\eqref{eqnQD}. 
 Since the right-hand side of~\eqref{eqnQD} is clearly locally absolutely continuous as a function of $x$, it can be used as a regularized quasi-derivative. 
  
 We will now use the latter equivalent formulation in order to prove a basic existence and uniqueness result for the inhomogeneous differential equation~\eqref{eqnDEinho}. 

\begin{lemma}\label{lemEE}
 For every $d_1$, $d_2\in\C$ there is a unique solution $f$ of the inhomogeneous differential equation~\eqref{eqnDEinho} with the initial conditions
 \begin{align}\label{eqnDEiv}
 f(0) = d_1 \qquad\text{and}\qquad f'\NL  = d_2.
 \end{align} 
 If $\chi$ is real-valued as well as $z$, $d_1$, $d_2\in\R$, then the solution $f$ is real-valued too. 
\end{lemma}

\begin{proof} 
 If $f$ is a solution of~\eqref{eqnDEinho} with the initial conditions~\eqref{eqnDEiv}, then the function  
 \begin{align*}
  F = \begin{pmatrix}  f \\ f' + \Nr_z f + \Qr \end{pmatrix} 
 \end{align*}
 has a locally absolutely continuous representative on $[0,L)$ in view of~\eqref{eqnQD}. 
 Moreover, this representative is clearly a solution of the initial value problem 
  \begin{align}\label{eqnDEequivsyst}
  F' & = \begin{pmatrix}  -\Nr_z & 1 \\ -\Nr^2_z & \Nr_z \end{pmatrix} F - \Qr \begin{pmatrix} 1 \\ \Nr_z \end{pmatrix}, & F(0) & = \begin{pmatrix} d_1 \\ d_2 \end{pmatrix}.
 \end{align}
 Since the coefficients of this first order system are locally integrable on $[0,L)$, the initial value problem~\eqref{eqnDEequivsyst} has a unique solution  (see, for example, \cite[Chapter~1]{ze05}), which proves the uniqueness part of the claim. 
 Conversely, if $F$ is a solution of the initial value problem~\eqref{eqnDEequivsyst}, then it is readily verified that $F_1$ is a solution of~\eqref{eqnDEinho} with the initial conditions~\eqref{eqnDEiv}, which proves the existence part of the claim.  
\end{proof}

In order to provide a representation of solutions to the inhomogeneous differential equation~\eqref{eqnDEinho}, we also consider the  corresponding homogeneous equation 
 \begin{align}\label{eqnDEho}
  -f'' = z\, \omega f + z^2 \dip f.
\end{align}
The Wronski determinant $W(\theta,\phi)$ of two solutions $\theta$, $\phi$ of this equation is defined as the number   
\begin{align}
W(\theta,\phi) = \theta(0)\phi'\NL-\theta'\NL\phi(0).
\end{align}

\begin{corollary}
 If $\theta$, $\phi$ are two solutions of the homogeneous differential equation~\eqref{eqnDEho}, then the Wronski determinant $W(\theta,\phi)$ is non-zero if and only if the functions $\theta$ and $\phi$ are linearly independent.
 Moreover, one has 
\begin{align}\label{eqnWronskiconst}
    \theta(x)\phi'(x)-\theta'(x)\phi(x)  = W(\theta,\phi)
\end{align}
for almost all $x\in[0,L)$. 
\end{corollary}

\begin{proof}
  It is readily verified using Lemma~\ref{lemEE} that  the Wronski determinant $W(\theta,\phi)$ is non-zero if and only if $\theta$ and $\phi$ are linearly independent.
  Moreover, applying \cite[Theorem~1.2.2]{ze05} to the equivalent system~\eqref{eqnDEequivsyst} in the proof of Lemma~\ref{lemEE} shows that  
  \begin{align*}
   \theta(\phi'+\Nr_z \phi) - (\theta'+\Nr_z\theta)\phi  
  \end{align*}
  is constant on $[0,L)$ and equal to $W(\theta,\phi)$, which proves the claim. 
\end{proof}


Two linearly independent solutions $\theta$, $\phi$ of the homogeneous differential equation~\eqref{eqnDEho} are called {\em a fundamental system} if their Wronski determinant $W(\theta,\phi)$ equals one.  
One notes that fundamental systems always exist due to Lemma~\ref{lemEE}. 
  
\begin{corollary}\label{cor:inhom}
 If $\theta$, $\phi$ are a fundamental system of the homogeneous differential equation~\eqref{eqnDEho}, then any solution $f$ of the inhomogeneous differential equation~\eqref{eqnDEinho} can be written as 
\begin{align}\label{eq:inhom}
 f(x) = d_1 \theta(x) + d_2 \phi(x) - \int_0^x \Qr(t) \left(\theta(x)\phi'(t) - \theta'(t)\phi(x)\right)dt, \quad x\in[0,L),
\end{align}
 for some constants $d_1$, $d_2\in\C$. 
\end{corollary}

\begin{proof}
 Upon noting that the matrix function 
 \begin{align*}
  \begin{pmatrix} \theta & \phi \\ \theta'+\Nr_z\theta & \phi'+\Nr_z\phi \end{pmatrix} 
 \end{align*}
 is a fundamental matrix for the homogeneous first order system corresponding to~\eqref{eqnDEequivsyst}, the claim follows from the usual variation of parameters formula (see, for example, \cite[Theorem~1.3.1]{ze05}). 
\end{proof}


As a final result of this section, we show that the solutions of the inhomogeneous differential equation~\eqref{eqnDEinho} with fixed initial conditions of the form~\eqref{eqnDEiv} depend analytically on the complex spectral parameter $z\in\C$. 

\begin{corollary}\label{corSolEnt}
 Fix $d_1$, $d_2\in\C$ and  for every $z\in\C$ let $f_z$ be the solution of the inhomogeneous differential equation~\eqref{eqnDEinho} with the initial conditions~\eqref{eqnDEiv}.  
 Then for every $g\in L^2[0,L)$ with compact support, the function
 \begin{align}\label{eq:fz}
   z\mapsto \int_0^L g(x)f_z'(x)dx 
  \end{align}
  is entire. 
  In particular, for every $x\in[0,L)$ the function $z\mapsto f_z(x)$ is entire and locally uniformly bounded as long as $x$ varies in compact subsets of $[0,L)$.  
\end{corollary}

\begin{proof}
  If $F_z$ denotes the solution of the corresponding initial value problem for the equivalent system~\eqref{eqnDEequivsyst}, then the function $z\mapsto F_z(x)$ is clearly entire for every $x\in[0,L)$. 
  Moreover, locally uniform bounds are available (see, for example, \cite[Theorem~1.5.1]{ze05}) as long as $x$ varies in compact subsets of $[0,L)$. 
  Since we have 
  \begin{align*}
  \int_0^L g(x)f_z'(x)dx = \int_0^L g(x) \left(F_{z,2}(x) - \Nr_z(x)F_{z,1}(x) - \Qr(x)\right)dx, \quad z\in\C, 
  \end{align*}
  this shows that the function in~\eqref{eq:fz} is entire as well. 
\end{proof}

  \section{Realization as a self-adjoint linear relation}

 In the present section we will introduce a self-adjoint linear relation associated with the differential equation~\eqref{eqnDEinho}. 
 To this end, we first consider the space 
 \begin{align}
 \cH = \Hasto\times L^2([0,L);\dip),
\end{align}
which turns into a Hilbert space when equipped with the usual scalar product
\begin{align}
 \spr{f}{g}_{\cH} = \int_0^L f_1'(x) g_1'(x)^\ast dx + \int_{[0,L)} f_2(x) g_2(x)^\ast d\dip(x), \quad f,\, g\in \cH.
\end{align} 
 The respective components of some function $f\in\cH$ are always denoted by adding subscripts, that is, with $f_1$ and $f_2$. 
 For future purposes, we note that point evaluations of the first component are clearly continuous on $\cH$ in view of~\eqref{eqnFuncHastGrow}. 
 Given some $x\in[0,L)$, we will denote with $\delta_{x}$ the unique function in $\cH$ such that 
 \begin{align}\label{eqnPEdeltacpm}
  \spr{f}{\delta_{x}}_{\cH} = f_1(x), \quad f\in\cH. 
 \end{align}
 It is readily verified that this function is explicitly given by 
\begin{align}
 \delta_x(t) =  \begin{pmatrix} 1 \\ 0 \end{pmatrix} \begin{cases} t \left(1-\frac{x}{L}\right), & t\in[0,x), \\ x\left(1-\frac{t}{L}\right), & t\in[x,L). \end{cases}
\end{align}

We now introduce the linear relation $\T$ in the Hilbert space $\cH$ by saying that some pair $(f,g)\in\cH\times\cH$ belongs to $\T$ if and only if 
\begin{align}\label{eqnDEre1}
-f_1'' & =\omega g_{1} + \dip g_{2}, &  \dip f_2 & =\dip g_{1},
\end{align}
holds. 
In order to be precise, the right-hand side of the first equation in~\eqref{eqnDEre1} has to be understood as the $H^{-1}_{\loc}[0,L)$ distribution given by 
\begin{align}
 h \mapsto \omega(g_1h) + \int_{[0,L)} g_2 h\, d\dip. 
\end{align} 
 Moreover, let us point out explicitly that the second equation in~\eqref{eqnDEre1} holds if and only if $f_2(x)=g_1(x)$ for almost all $x\in[0,L)$ with respect to $\dip$. 

Although the linear relation $\T$ is in general not (the graph of) an operator, we will employ the following convenient notation for elements of $\T$: 
Given a pair $\f\in\T$, we will denote its first component with $f$ and its second component with $\tau f$. 

 \begin{theorem}\label{thmTsa}
  The linear relation $\T$ is self-adjoint in $\cH$. 
 \end{theorem}
 
\begin{proof}
 Let us begin with proving that $\T$ is symmetric in $\cH$, that is, we have 
 \begin{align*}
      \spr{\tau f}{g}_{\cH}=  \spr{f}{\tau g}_{\cH}, \quad \f,\,\g\in\T. 
 \end{align*}
To this end, we let $\f\in\T$ and first introduce the auxiliary function 
\begin{align*}
 f_1^\qd(x) = f_1'\NL + \int_0^x \Wr(t) \tau f_1'(t)dt - \int_{[0,x)} \tau f_2\, d\dip, \quad x\in[0,L),  
\end{align*}
where $\Wr$ denotes the normalized anti-derivative of $\omega$. 
From the differential equation which $f_1$ satisfies, cast in the form~\eqref{eqnQD}, we then get 
\begin{align}\label{eqnTsaQD}
 f_1'(x) +  \Wr(x)\tau f_1(x) = f_1^\qd(x)  
\end{align}
for almost all $x\in[0,L)$. 
Employing this and an integration by parts (use, for example, \cite[Exercise~5.8.112]{bo07}, \cite[Theorem~21.67]{hest65}), one obtains for every $h\in\cH$
\begin{align*}
 \int_0^x  h_1'(t) f_1'(t) dt & = h_1f_1^\qd|_{0}^x - \int_0^x \Wr(t) (h_1\tau f_1)'(t)dt + \int_{[0,x)} h_1 \tau f_2\, d\dip, \quad x\in[0,L).
\end{align*}
Now let $\g\in\T$, choose $h=\tau g^\ast$ in the equation above, subtract the corresponding equation with the roles of $\f$ and $\g^\ast$ interchanged, take also into account the second equation in~\eqref{eqnDEre1} and finally let $x\rightarrow L$ to end up with  
 \begin{align*}
   \spr{\tau f}{g}_{\cH} = \spr{f}{\tau g}_{\cH} +  \lim_{x\rightarrow L} \tau f_1(x)g_1^\qd(x)^\ast - f_1^\qd(x) \tau g_1(x)^\ast.
 \end{align*}
 We are left to prove that the limit (which is already known to exist) is actually zero. 
 In view of~\eqref{eqnTsaQD} and the corresponding equation for $\g$ as well as the estimate~\eqref{eqnFuncHastGrow} for $\tau f_1$, $\tau g_1$ one infers that the function    
 \begin{align*}
  \frac{|\tau f_1(x)g_1^\qd(x)^\ast-f_1^\qd(x)\tau g_1(x)^\ast|^2}{x\left(1-\frac{x}{L}\right)}, \quad x\in(0,L), 
 \end{align*}
 is integrable near $L$ and it remains to note that 
 this guarantees the existence of an increasing sequence $\lbrace x_n\rbrace_{n=1}^\infty$ in $[0,L)$ with $x_n\rightarrow L$ as $n\rightarrow\infty$ such that 
 \begin{align*}
  \tau f_1(x_n)g_1^\qd(x_n)^\ast-f_1^\qd(x_n)\tau g_1(x_n)^\ast \rightarrow 0  
 \end{align*}
 as $n\rightarrow\infty$, which proves that $\T$ is symmetric in $\cH$.

 In order to show that $\T$ is even self-adjoint, we fix some $(f,f_\tau)\in\T^\ast$ and let $h\in\cH$. 
 We are first going to show that there is a $g\in\cH$ such that $(g,h)\in\T$ provided that $h$ vanishes near $L$. 
 In fact, the first component of this $g$ is given by 
 \begin{align*}
  g_1(x) & =d x - \int_0^x \Hr(t)dt, \quad x\in[0,L),  & d & =  \lim_{x\rightarrow L} \frac{1}{x}\int_0^x \Hr(t)dt,
 \end{align*}
 where $\Hr$ is the normalized anti-derivative of $\omega h_1 + \dip h_2$. 
 Since $\Hr$ is constant near $L$ (due to the fact that $h$ vanishes near $L$), the function $g_1$ is well-defined, belongs to $\Hasto$ and is readily verified to satisfy the differential equation
 \begin{align*}
  -g_1'' = \omega h_1 + \dip h_2 
 \end{align*}
 in view of~\eqref{eqnDEinhoIntNAD}. 
 Upon setting $g_2(x) = h_1(x)$ for almost all $x\in[0,L)$ with respect to $\dip$, we conclude that $(g,h)\in\T$. 
 As a consequence, we have 
 \begin{align*}
  \spr{f}{h}_{\cH} & = \spr{f_\tau}{g}_{\cH}  = -\int_0^L \Wr(x) (f_{\tau,1} h_1^\ast)'(x)dx + \int_{[0,L)}  f_{\tau,2} h_1^\ast + f_{\tau,1} h_2^\ast\, d\dip
 \end{align*}
 where we used the identity~\eqref{eqnTsaQD} rewritten for $(g,h)\in\T$ as well as an integration by parts, also taking into account the second equation in~\eqref{eqnDEre1}.  
 Since the above equality holds for arbitrary functions $h\in\cH$ vanishing near $L$, we infer that 
 \begin{align}\label{eqnRefftau}
  -f_1'' & = \omega f_{\tau,1} + \dip f_{\tau,2}, & \dip f_2 & = \dip f_{\tau,1}. 
 \end{align}
 In fact, upon taking functions $h$ for which $h_2$ vanishes identically, one arrives at the differential equation in~\eqref{eqnRefftau} cast in the form used in Definition~\ref{defSolution}. 
 Let us note that it does not matter that all our test functions $h_1$ here vanish at zero, opposed to the ones in Definition~\ref{defSolution}.
 On the other side, upon taking functions $h$ for which $h_1$ vanishes identically, one readily infers the second equation in~\eqref{eqnRefftau}.  
 Of course, this means nothing but $(f,f_\tau)\in\T$ and thus concludes the proof. 
 \end{proof}
 
 The linear relation $\T$ is indeed closely related to the differential equation~\eqref{eqnDEinho}. 
For any $z\in\C$, a pair $(f,g)\in\cH\times\cH$ belongs to $\T-z$ if and only if 
\begin{align}\label{eqnTloc-z}
 -f_1'' & = z\, \omega f_1 +z^2 \dip f_1 + \omega g_1 + z\,\dip g_1 + \dip g_2, & \dip f_2 = z\,\dip f_1 +  \dip g_1. 
\end{align}
This observation shows that some $f\in\cH$ belongs to $\ker{\T-z}$ if and only if $f_1$ is a solution of the homogeneous differential equation~\eqref{eqnDEho} and $\dip f_2 = z\, \dip f_1$. 
In other words, some $z\in\C$ is an eigenvalue of $\T$ if and only if there is a non-trivial solution $\phi$ of the homogeneous differential equation~\eqref{eqnDEho} such that $\phi$ lies in $\Hasto$ and $z\phi$ lies in $L^2([0,L);\dip)$.
 Furthermore, it is readily deduced from these facts that every eigenvalue of $\T$ is non-zero and simple in view of Lemma~\ref{lemEE}. 
   
   \begin{lemma}\label{lemWS}
   If $z$ belongs to the resolvent set of $\T$, then there is a (up to scalar multiples) unique non-trivial solution $\psi$ of the homogeneous differential equation~\eqref{eqnDEho} such that $\psi$ lies in $H'[0,L)$ and $z\psi$ lies in $L^2([0,L);\dip)$. 
  \end{lemma}
 
  \begin{proof}
   We consider the auxiliary space 
   \begin{align*}
    \cK = H'[0,L) \times L^2([0,L);\dip), 
   \end{align*}
   which turns into a Hilbert space when equipped with the scalar product 
    \begin{align*}
  \spr{f}{g}_{\cK} =  f_1(0)g_1(0)^\ast + \int_0^L f_1'(x) g_1'(x)^\ast dx + \int_{[0,L)} f_2(x) g_2(x)^\ast d\dip(x), \quad f,\,g\in\cK.
  \end{align*}
   Clearly, the space $\cH$ can be regarded as a closed subspace of $\cK$ and its orthogonal complement $\cH^\bot$ is spanned by a single function $k\in\cK$ which is explicitly given by 
   \begin{align*}
     k(x) = \left(1 - \frac{x}{L}\right) \begin{pmatrix} 1\\0\end{pmatrix}, \quad x\in[0,L). 
   \end{align*} 
   In view of this decomposition, the linear relation
   \begin{align*}
    \T_0 = \T \oplus\left(\cH^\bot\times\lbrace0\rbrace\right)
   \end{align*}
   is self-adjoint in $\cK$. 
   Moreover, it is a one-dimensional restriction, 
 \begin{align*}
  \T_0 = \left\lbrace \f\in\T_{\mathrm{max}} \,|\, \tau f_1(0) = 0 \right\rbrace, 
 \end{align*}   
   of the linear relation $\T_{\mathrm{max}}$ in $\cK$, defined by saying that some pair $(f,g)\in\cK\times\cK$ belongs to $\T_{\mathrm{max}}$ if and only if~\eqref{eqnDEre1} holds.  
 As a consequence, the deficiency indices of the symmetric linear relation $\T_{\mathrm{max}}^\ast$ are equal to one.
 This guarantees that 
 \begin{align*}
  \dim\ker{\T_{\mathrm{max}}-z} = 1 
 \end{align*}
 for all $z$ in the resolvent set of $\T_0$, which readily proves the claim.
  \end{proof}
 
  The existence of a solution $\psi$ of the homogeneous differential equation~\eqref{eqnDEho}  such that $\psi$ lies in $H'[0,L)$ and $z\psi$ lies in $L^2([0,L);\dip)$ as in Lemma~\ref{lemWS} can also be established employing the usual Weyl circles method (see, for example, \cite[\S 2.5]{kakr74}). 
   
  \begin{proposition}\label{propRes}
  If $z$ belongs to the resolvent set of $\T$, then one has   
  \begin{align}\begin{split}\label{eq:Res}
   z\, (\T-z)^{-1} g(x) & = \spr{g}{\G(x,\cdot\,)^\ast}_{\cH} \begin{pmatrix} 1 \\ z \end{pmatrix} - g_1(x) \begin{pmatrix} 1 \\ 0 \end{pmatrix}, \quad x\in[0,L),  
 \end{split}\end{align} 
 for every $g\in\cH$, where the Green's function $\G$ is given by  
 \begin{align}
  \G(x,t) =  \begin{pmatrix} 1 \\ z \end{pmatrix} \frac{1}{W(\psi,\phi)} \begin{cases} \psi(x) \phi(t), & t\in[0,x), \\ \psi(t) \phi(x), & t\in[x,L),    \end{cases}
 \end{align}
 and $\psi$, $\phi$ are linearly independent solutions of the homogeneous differential equation~\eqref{eqnDEho} such that $\phi$ vanishes at zero, $\psi$ lies in $H'[0,L)$ and $z\psi$ lies in $L^2([0,L);\dip)$. 
 \end{proposition}

\begin{proof}
 First of all, we note that non-trivial solutions $\psi$, $\phi$ of~\eqref{eqnDEho} with the required properties always exist due to Lemma~\ref{lemEE} and Lemma~\ref{lemWS}. 
 These solutions are necessarily linearly independent since otherwise $z$ would be an eigenvalue of $\T$. 
 Furthermore, let us also recall that for every solution $\theta$ of~\eqref{eqnDEho} we have 
 \begin{align}\label{eqnDEThetaInt}
  \theta'(x) + z\,\Wr(x)\theta(x) = \theta^\qd(x) := \theta'\NL + z\int_0^x \Wr(t)\theta'(t)dt - z^2 \int_{[0,x)} \theta\, d\dip
 \end{align}
 for almost all $x\in[0,L)$. 
 Now we fix some $(f,g)\in\T-z$, introduce the function
  \begin{align*}
  f_1^\qd(x) = f_1'\NL + \int_0^x \Wr(t)(zf_1'(t) + g_1'(t))dt - \int_{[0,x)} z^2f_1 + zg_1 + g_2\,d\dip, \quad x\in[0,L), 
 \end{align*}
  and note that the differential equation in~\eqref{eqnTloc-z} yields
 \begin{align}\label{eqnTloc-zInt}
  f_1'(x) + z\,\Wr(x)f_1(x) + \Wr(x)g_1(x) = f_1^\qd(x)
 \end{align}
 for almost all $x\in[0,L)$.
 Given some $x\in[0,L)$, we use~\eqref{eqnDEThetaInt} and an integration by parts to obtain (here, $\G_1$ is differentiated with respect to the second variable) 
  \begin{align*}
   & \int_0^r \G_1'(x,t) g_1'(t) dt + \int_{[0,r)} \G_2(x,\redot) g_2\,d\dip - g_1(x) - \frac{\phi(x)}{W(\psi,\phi)} \psi^\qd(r)g_1(r) \\
  & \qquad = - z\int_0^r \Wr(t) (\G_1(x,\redot)g_1)'(t)dt + z \int_{[0,r)} \G_1(x,\redot)(zg_1+g_2)\,d\dip,   
   \end{align*}  
 which holds as long as $r\in[x,L)$. 
 Moreover, using~\eqref{eqnTloc-zInt} and another integration by parts, the above equation is furthermore equal to   
 \begin{align*}
  & z\int_0^r \G_1'(x,t)f_1'(t)dt  - \frac{\phi(x)}{W(\psi,\phi)} z\psi(r)f_1^\qd(r) \\ 
  & \qquad\qquad  + z^2\int_0^r \Wr(t)(\G_1(x,\redot)f_1)'(t)dt - z^3\int_{[0,r)} \G_1(x,\redot)f_1\,d\dip,  
 \end{align*}
  and upon invoking~\eqref{eqnDEThetaInt} as well as a final integration by parts, we  end up with 
  \begin{align*}
   z f_1(x) + \frac{\phi(x)}{W(\psi,\phi)} \left(\psi^\qd(r)zf_1(r) - z\psi(r)f_1^\qd(r)\right).  
 \end{align*}
  Altogether, upon letting $r\rightarrow L$, these equations give 
 \begin{align*}
   z f_1(x) & = \spr{g}{\G(x,\cdot\,)^\ast}_{\cH} - g_1(x) \\ & \qquad\qquad + \frac{\phi(x)}{W(\psi,\phi)} \lim_{r\rightarrow L} \left( z\psi(r)f_1^\qd(r) - \psi^\qd(r)(zf_1(r)+g_1(r)) \right).
 \end{align*}  
 Now one infers as in the first part of the proof of Theorem~\ref{thmTsa} that the limit is actually zero, which gives the claimed representation of $zf_1$ as in~\eqref{eq:Res}.  
Finally, as an immediate consequence of the definition of $\T$ we have 
\begin{align*}
 zf_2(x)  = z\left(z f_1(x) + g_1(x)\right)= z\,\spr{g}{\G(x,\redot)^\ast}_{\cH},
\end{align*}
for almost all $x\in[0,L)$ with respect to $\dip$. 
\end{proof}

  \section{The principal Weyl--Titchmarsh function}\label{secWT}
 
 Associated with our spectral problem is  the Weyl--Titchmarsh function $m$ defined on $\C\backslash\R$ by 
 \begin{align}
  m(z) =  \frac{\psi'\NLz}{z\psi(z,0)},\quad z\in\C\backslash\R,
 \end{align} 
 where $\psi(z,\redot)$ is a non-trivial solution of the homogeneous differential equation~\eqref{eqnDEho} which lies in $H'[0,L)$ and $L^2([0,L);\dip)$ as in Lemma~\ref{lemWS}. 
 Note that the denominator in this definition does not vanish since otherwise $z$ would be an eigenvalue of $\T$. 

 For every $z\in\C$, we introduce the fundamental system of solutions $\theta(z,\redot)$, $\phi(z,\redot)$ of the homogeneous differential equation~\eqref{eqnDEho} satisfying the initial conditions
 \begin{align}
  \theta(z,0)& = \phi'\NLz =1, &  \theta'\NLz & = \phi(z,0) =0.
 \end{align}
 It follows readily from the very definition of $m$ that the solution  
 \begin{align}\label{eq:m_pm}
   \psi(z,x) = \theta(z,x) +  m(z)\,z\phi(z,x), \quad x\in[0,L), 
 \end{align}
 of~\eqref{eqnDEho} lies in $H'[0,L)$ and $L^2([0,L);\dip)$ for every $z\in\C\backslash\R$.

 Let us recall that a function $m$ on $\C\backslash\R$ is called {\it a Herglotz--Nevanlinna function} if it is analytic, maps the upper complex half-plane into the closure of the upper complex half-plane and satisfies the symmetry relation
 \begin{align}\label{eqnHNsym}
  m(z)^\ast = m(z^\ast), \quad z\in\C\backslash\R. 
 \end{align}

\begin{lemma}\label{lemHNm}
The Weyl--Titchmarsh function $m$ is a Herglotz--Nevanlinna function. 
\end{lemma}

\begin{proof}
 From Proposition~\ref{propRes} we infer that for every $x\in[0,L)$ one has  
\begin{align}\begin{split}\label{eqnRelResm}
 z\, & \spr{(\T-z)^{-1}\delta_{x}}{\delta_{x}}_{\cH} +  x\left(1-\frac{x}{L}\right)  \\ 
    & \qquad\qquad=  \left(\theta(z,x) + m(z)\, z \phi(z,x) \right) \phi(z,x), \quad z\in\C\backslash\R.
\end{split}\end{align}
Since for every $z\in\C$ we are able to find an $x\in[0,L)$ such that $\phi(z,x)$ does not vanish, this equation shows that the function $m$ is analytic on $\C\backslash\R$ in view of Corollary~\ref{corSolEnt}. 
Moreover, from Lemma~\ref{lemEE} we deduce that  
\begin{align*}
 \theta(z,x)^\ast & = \theta(z^\ast,x), & \phi(z,x)^\ast & = \phi(z^\ast,x),
\end{align*}
 for all $z\in\C$ and $x\in[0,L)$, which readily implies the symmetry relation~\eqref{eqnHNsym}. 
Finally, the fact that $m$ is a Herglotz--Nevanlinna function follows from the identity
\begin{align*}
 \frac{m(z) - m(z)^{\ast}}{z-z^{\ast}} & = \frac{1}{|z|^2} \int_0^L \left|\psi'(z,x)\right|^2 dx + \int_{[0,L)} \left|\psi(z,x)\right|^2 d\dip(x), \quad z\in\C\backslash\R, 
\end{align*}
 where $\psi$ is given by~\eqref{eq:m_pm} and we used~\eqref{eqnDEThetaInt}, integration by parts 
 as well as  
 \begin{align*}
  \lim_{x\rightarrow L} z\psi(z,x)\psi^\qd(z^\ast,x) - \psi^\qd(z,x) z^\ast\psi(z^\ast,x) = 0, \quad z\in\C\backslash\R, 
 \end{align*}
 which can be proved as in the first part of the proof of Theorem~\ref{thmTsa}. 
\end{proof}
 
 Let us mention that the Weyl--Titchmarsh function $m$ can also be introduced in a different way using just the fundamental system of solutions $\theta$ and $\phi$; cf.\  \cite[\S 10.4]{kakr74}. 
 The proof is fairly standard but we shall present it for the sake of completeness. 
 
  \begin{lemma}\label{lemWTasQ}
  The Weyl--Titchmarsh function $m$ is given by 
  \begin{align}\label{eq:mFalt}
    m(z) = \lim_{x\rightarrow L} -\frac{\theta(z,x)}{z\phi(z,x)}, \quad z\in\C\backslash\R,
  \end{align}
  where the convergence is locally uniformly. 
 \end{lemma}

 \begin{proof}
  For every $x\in(0,L)$, we first consider the function
  \begin{align*}
   m_x(z) = - \frac{\theta(z,x)}{z\phi(z,x)}, \quad z\in\C\backslash\R,
  \end{align*}
  which is well-defined and analytic.
  Indeed, if $\phi(z,x)$ was zero, then integration by parts using~\eqref{eqnDEThetaInt} would imply the contradiction 
  \begin{align}\label{eqnphizerocont}
   \int_0^x |\phi'(z,t)|^2 dt + \int_{[0,x)} |z\phi(z,t)|^2 d\dip(t) = 0
  \end{align}
  as long as $z$ is non-real. 
  Upon introducing the solutions
  \begin{align*}
   \psi_x(z,t) = \theta(z,t) + m_x(z)\, z\phi(z,t), \quad t\in[0,L),~z\in\C\backslash\R, 
  \end{align*}
  one shows (as in the proof of Lemma~\ref{lemHNm}) via integration by parts using~\eqref{eqnDEThetaInt} that 
  \begin{align}\label{eqnHNmx}
   \frac{m_x(z)-m_x(z)^\ast}{z-z^\ast} = \frac{1}{|z|^{2}} \int_0^x |\psi_x'(z,t)|^2 dt + \int_{[0,x)} |\psi_x(z,t)|^2 d\dip(t), \quad z\in\C\backslash\R.
  \end{align}
  In particular, this means that $m_x$ is a Herglotz--Nevanlinna function. 
  
  Now let $\lbrace x_n\rbrace_{n=1}^\infty$ be a sequence in $(0,L)$ with $x_n\rightarrow L$ as $n\rightarrow\infty$ and suppose that $m_{x_n}$ converges locally uniformly either to a function $Q$ or to $\infty$. 
 By the fundamental normality test (see \cite[Section~2.7]{sc93}), it suffices to show that the first case prevails and that $Q$ coincides with the Weyl--Titchmarsh function $m$.
 To this end, we first note that~\eqref{eqnHNmx} yields the inequality 
  \begin{align}\label{eqnHNmxest}
   \frac{1}{|z|^2} \int_0^x |\psi_{x_n}'(z,t)|^2 dt +  \int_{[0,x)} |\psi_{x_n}(z,t)|^2 d\dip(t) \leq \frac{m_{x_n}(z)-m_{x_n}(z)^\ast}{z-z^\ast} 
  \end{align}
  for every fixed $z\in\C\backslash\R$ and $x\in[0,L)$ as long as $n\in\N$ is large enough. 
  If the functions $m_{x_n}$ converged locally uniformly to $\infty$, then letting $n\rightarrow\infty$ in~\eqref{eqnHNmxest} after dividing by $|m_{x_n}(z)|^2$, would yield the contradiction~\eqref{eqnphizerocont}. 
  Thus, the sequence $m_{x_n}$ converges locally uniformly to a function $Q$ and in order to show that $Q$ coincides with the Weyl--Titchmarsh function $m$, we simply need to show that the function
  \begin{align*}
   \psi_{L}(z,t) = \theta(z,t) + Q(z)\, z \phi(z,t), \quad t\in[0,L),
  \end{align*}
  lies in $H'[0,L)$ and $L^2([0,L);\dip)$ for every $z\in\C\backslash\R$. 
  Letting $n\rightarrow\infty$ in~\eqref{eqnHNmxest}, we get 
  \begin{align*}
   \frac{1}{|z|^2} \int_0^x |\psi_{L}'(z,t)|^2 dt +  \int_{[0,x)} |\psi_{L}(z,t)|^2 d\dip(t) \leq \frac{Q(z)-Q(z)^\ast}{z-z^\ast}, \quad z\in\C\backslash\R,
  \end{align*}
  which concludes the proof since $x\in[0,L)$ was arbitrary. 
 \end{proof}
 
 \begin{remark}\label{rem:dualStr}
 Comparing our definition of the Weyl--Titchmarsh function with the definition used by M.\ G.\ Krein in the case when $\omega$ is a non-negative Borel measure on $[0,L)$ and $\dip$ vanishes identically (see \cite[Theorem 10.1]{kakr74}), let us mention that the Weyl--Titchmarsh function given by~\eqref{eq:mFalt} coincides with the coefficient of dynamic compliance of the corresponding dual string in this case (see \cite[Equation~(12.5)]{kakr74}).
 \end{remark}
 
As a Herglotz--Nevanlinna function, the Weyl--Titchmarsh function $m$ clearly has an integral representation \cite{kakr74a}, \cite[Section~5.3]{roro94} of the form  
\begin{align}\label{eqnWTmIntRep}
 m(z) = c_1 z + c_2 - \frac{1}{Lz} +  \int_\R \frac{1}{\lambda-z} - \frac{\lambda}{1+\lambda^2}\, d\mu(\lambda), \quad z\in\C\backslash\R, 
\end{align}
for some constants $c_1$, $c_2\in\R$ with $c_1\geq0$ 
and a non-negative Borel measure $\mu$ on $\R$ with $\mu(\lbrace0\rbrace)=0$ for which the integral  
\begin{align}
 \int_\R \frac{d\mu(\lambda)}{1+\lambda^2}
\end{align}
is finite. 
For reasons that will become clear soon we singled out a possible singularity of $m$ at zero in~\eqref{eqnWTmIntRep}, that is, we removed a possible point mass from $\mu$ at zero. 
The fact that $m$ can be written in the particular form in~\eqref{eqnWTmIntRep} follows from  
\begin{align}
 - \lim_{\varepsilon\downarrow0} \I\varepsilon\, m(\I\varepsilon) = \frac{1}{L},  
\end{align}
where we used the identity~\eqref{eqnRelResm} to compute the limit. 

The measure $\mu$, which can be recovered from $m$ as usual via the Stieltjes inversion formula (see, for example, \cite[\S 2]{kakr74a}), will turn out to be a {\em spectral measure} for the (operator part of the) linear relation $\T$.  
 To this end, let us first define the transform  
 \begin{align}\begin{split}\label{eqnFhat}
   \hat{f}(z) = \int_0^L \phi'(z,x) f_1'(x)dx + \int_{[0,L)} z \phi(z,x)f_2(x) d\dip(x), \quad z\in\C.  
 \end{split}\end{align}
 for every function $f\in\cH$ with compact support in $[0,L)$.
 One notes that this transform is an entire function (due to Corollary~\ref{corSolEnt}) which vanishes at zero.
 
  In order to state the next result, recall that for all $f$, $g\in\cH$ there is a unique complex Borel measure $E_{f,g}$ on $\R$ such that
\begin{align}\label{eqnLDstietrans}
  \spr{ (\T-z)^{-1} f}{g}_{\cH} = \int_\R \frac{dE_{f,g}(\lambda)}{\lambda-z}, \quad z\in\C\backslash\R.
\end{align}

\begin{lemma}\label{lemSMst}
 Given functions $f$, $g\in\cH$ with compact support in $[0,L)$, we have 
\begin{align}
 E_{f,g}(B) = \int_{B} \hat{f}(\lambda)  \hat{g}(\lambda)^\ast d\mu(\lambda)
\end{align}
for every Borel set $B\subseteq\R$.
\end{lemma}

\begin{proof}
Using Proposition~\ref{propRes}, a lengthy but straightforward calculation gives  
\begin{align*}
 \spr{(\T-z)^{-1} f}{g}_{\cH} = m(z) \hat{f}(z) \hat{g}(z^\ast)^\ast + H_{f,g}(z), \quad z\in\C\backslash\R,
\end{align*}
for some entire function $H_{f,g}$ and thus for all $\lambda_1$, $\lambda_2\in\R$ with $\lambda_1<\lambda_2$
 \begin{align*}
  E_{f,g}([\lambda_1,\lambda_2)) & = \lim_{\delta\downarrow0} \lim_{\varepsilon\downarrow0} \frac{1}{\pi} \int_{\lambda_1-\delta}^{\lambda_2-\delta}  \spr{\im\,(\T-\lambda-\I\varepsilon)^{-1} f}{g}_{\cH} \,d\lambda \\ 
  & = \lim_{\delta\downarrow0} \lim_{\varepsilon\downarrow0} \frac{1}{\pi} \int_{\lambda_1-\delta}^{\lambda_2-\delta} \hat{f}(\lambda) \hat{g}(\lambda)^\ast \,\im\, m(\lambda+\I\varepsilon) \,d\lambda \\
  & = \int_{[\lambda_1,\lambda_2)} \hat{f}(\lambda) \hat{g}(\lambda)^\ast d\mu(\lambda),
 \end{align*}
 where we used the weak version of Stone's formula (also note $\hat{f}(0)=\hat{g}(0)=0$).
\end{proof}

In particular, the preceding lemma shows that the mapping $f\mapsto \hat{f}$ uniquely extends to a bounded linear operator $\F$ from $\cH$ into $\Lmu$.
More precisely, for every function $f\in\cH$ with compact support we have 
\begin{align}
 \| \hat{f}\|^2_{\Lmu} = \int_\R \hat{f}(\lambda) \hat{f}(\lambda)^\ast d\mu(\lambda) = E_{f,f}(\R) =  \|\mathrm{P} f\|_{\cH}^2,
\end{align}
where $\mathrm{P}$ is the orthogonal projection onto the closure $\D$ of $\dom{\T}$. 
This shows that the transformation $\F$ is actually a partial isometry from $\cH$ into $\Lmu$ with initial subspace $\D$. 
Of course, the result of Lemma~\ref{lemSMst} now immediately extends to all functions $f$, $g\in\cH$, that is, 
\begin{align}\label{eqnSMF}
 E_{f,g}(B) =  \int_B \F f(\lambda) \F g(\lambda)^\ast d\mu(\lambda)
\end{align}
for every Borel set $B\subseteq\R$. 
It will turn out that the transformation $\F$ maps the (operator part of the) self-adjoint linear relation $\T$ to multiplication with the independent variable in $\Lmu$. 
Before we get to prove this, we first need to derive a few more properties of this transformation. 

\begin{lemma}\label{lemTransPE}
 For each $x\in[0,L)$ we have 
 \begin{align}
  \F\delta_x(\lambda) = \phi(\lambda,x)
 \end{align} 
 for almost all $\lambda\in\R$ with respect to $\mu$.
\end{lemma}

\begin{proof}
Firstly, we infer from equation~\eqref{eqnRelResm} that   
\begin{align*}
 \spr{(\T-z)^{-1} \delta_x}{\delta_x}_{\cH} = \left(m(z)+\frac{1}{Lz}\right) \phi(z,x)\phi(z,x) + H_{x,x}(z), \quad z\in\C\backslash\R,
\end{align*}
 for some entire function $H_{x,x}$ and as in the proof of Lemma~\ref{lemSMst} we conclude that 
\begin{align*}
 \|\mathrm{P}\delta_x\|_{\cH}^2 = \int_\R |\phi(\lambda,x)|^2 d\mu(\lambda).
\end{align*}
Secondly, given a function $f\in\cH$ with compact support, Proposition~\ref{propRes} gives     
\begin{align*}
 \spr{(\T-z)^{-1} f}{\delta_x}_{\cH} = m(z) \hat{f}(z) \phi(z,x) + H_{f,x}(z), \quad z\in\C\backslash\R,
\end{align*}
for some entire function $H_{f,x}$.  
Again we infer as in the proof of Lemma~\ref{lemSMst} that 
\begin{align*}
 \spr{f}{\mathrm{P}\delta_x}_{\cH} = \int_\R \F f(\lambda) \phi(\lambda,x) d\mu(\lambda),  
\end{align*} 
which extends to all functions $f\in\cH$ by continuity. 
This finally yields  
\begin{align*}
    \spr{\F\delta_x}{\phi(\ledot,x)}_{\Lmu} & = \spr{\delta_x}{\mathrm{P}\delta_x}_{\cH} = \|\F\delta_x\|_{\Lmu}^2 = \|\phi(\ledot,x)\|_{\Lmu}^2,
\end{align*} 
which proves the claim. 
\end{proof}

\begin{lemma}\label{lemFadjoint}
 The adjoint of the operator $\F$ is given by 
 \begin{align}
  \F^\ast g(x) = \lim_{r\rightarrow\infty} \int_{-r}^r \begin{pmatrix} 1 \\ \lambda \end{pmatrix} \phi(\lambda,x) g(\lambda) d\mu(\lambda), \quad x\in[0,L), 
 \end{align} 
 for every $g\in\Lmu$, where the limit has to be understood as a limit in $\cH$. 
\end{lemma}

\begin{proof}
 Suppose that the function $g\in\Lmu$ has compact support, set
 \begin{align*}
  \check{g}(x) = \int_\R \begin{pmatrix} 1 \\ \lambda \end{pmatrix} \phi(\lambda,x) g(\lambda) d\mu(\lambda), \quad x\in[0,L), 
 \end{align*}
 and note that $\check{g}_1$ belongs to $\Hasto$ since 
 \begin{align*}
   \check{g}_1(x) & 
                            = \spr{g}{\F\delta_x}_{\Lmu}  = \spr{\F^\ast g}{\delta_x}_{\cH}   
                            =  (\F^\ast g)_1(x), \quad x\in[0,L).
 \end{align*} 
 Now for arbitrary $c\in[0,L)$ one obtains upon interchanging integrals 
\begin{align*}
 I_{c}^{2} & = \int_{[0,c)} \left|\check{g}_2(x)\right|^2 d\dip(x)  
            =  \int_{[0,c)} \check{g}_2(x) \int_\R \lambda \phi(\lambda,x) g(\lambda)^\ast d\mu(\lambda)\, d\dip(x)  \\
          & = \int_\R g(\lambda)^\ast \int_{[0,c)} \lambda\phi(\lambda,x) \check{g}_2(x) d\dip(x)\, d\mu(\lambda) 
           = \int_\R  g(\lambda)^\ast \F\begin{pmatrix} 0 \\ \indik_{[0,c)} \check{g}_2 \end{pmatrix}(\lambda) \,  d\mu(\lambda) \\         
          & \leq \left\| g\right\|_{\Lmu} \left\| \F\begin{pmatrix} 0 \\ \indik_{[0,c)} \check{g}_2 \end{pmatrix}\right\|_{\Lmu}  
          \leq \left\| g\right\|_{\Lmu} I_{c},         
\end{align*}
 which shows that $\check{g}$ belongs to $\cH$. 
 Now if $f\in\cH$ is such that $f_1=0$ and $f_2$ has compact support, then upon interchanging integrals one sees 
 \begin{align*}
  \spr{\check{g}_2}{f_2}_{L^2([0,L);\dip)} = \spr{g}{\hat{f}}_{\Lmu} & = \spr{\F^\ast g}{f}_{\cH} = \spr{(\F^\ast g)_2}{f_2}_{L^2([0,L);\dip)},  
 \end{align*}
 implying $\check{g} = \F^\ast g$ and hence the claim. 
 \end{proof}

\begin{lemma}\label{lemSTonto}
 The mapping $\F$ is onto with (in general multi-valued) inverse 
 \begin{align}\label{eqnFinverse}
  \F^{-1} = \F^\ast \oplus \left(\lbrace0\rbrace \times \mul{\T}\right). 
 \end{align}
\end{lemma}

 \begin{proof}
  Let $\lambda_0\in\R$ and choose some $x\in[0,L)$ such that $\phi(\lambda_0,x)$ is not zero. 
  Then for every small enough neighborhood $U\subseteq\R$ around $\lambda_0$, the function
 \begin{align*}
  G(\lambda) = \begin{cases}
                \phi(\lambda,x)^{-1}, & \lambda\in U, \\
                0,                    & \lambda\in\R \backslash U,
               \end{cases}
 \end{align*}
 is bounded. 
 By a variant of the spectral theorem, there exists a $g\in\cH$ such that 
 \begin{align*}
   E_{g,\delta_c}(B) & = \int_B G(\lambda) dE_{\delta_x,\delta_c}(\lambda) 
 \end{align*}
 for every Borel set $B\subseteq\R$ and $c\in[0,L)$. 
 In view of~\eqref{eqnSMF} we conclude that $\F g(\lambda) = G(\lambda)\F \delta_x(\lambda) = \indik_U(\lambda)$ for almost all $\lambda\in\R$ with respect to $\mu$. 
 Thus the range of $\F$ contains all characteristic functions of bounded intervals, which shows that $\F$ is onto since the range of a partial isometry is always closed. 
 
 In order to verify~\eqref{eqnFinverse}, it suffices to note that $\F\F^\ast$ is the identity operator in $\Lmu$ and $\F^\ast\F$ is the orthogonal projection onto $\D=\mul{\T}^\bot$ in $\cH$. 
\end{proof}

 In the following we will denote with $\mathrm{M}_{\mathrm{id}}$ the maximally defined operator of multiplication with the independent variable in $\Lmu$.

\begin{theorem}\label{th:TsimM}
 The transformation $\F$ maps the self-adjoint linear relation $\T$ to multiplication with the independent variable in $\Lmu$, that is, 
 \begin{align}
  \F\,\T\, \F^\ast = \mathrm{M}_{\mathrm{id}}.
 \end{align} 
\end{theorem}

\begin{proof}
 First of all, we infer from~\eqref{eqnSMF} that for every $g\in\Lmu$ one has  
 \begin{align*}
  g\in\dom{\mathrm{M}_{\mathrm{id}}} 
              & \quad\Leftrightarrow\quad \F^\ast g\in\dom{\T} \quad\Leftrightarrow\quad  g\in\dom{\F\,\T\,\F^\ast}.
 \end{align*}
 In this case, equation~\eqref{eqnSMF} and~\cite[Lemma~B.4]{MeasureSL} show  that 
 \begin{align*}
   \spr{\mathrm{M}_{\mathrm{id}} g}{h}_{\Lmu} & = \int_\R \lambda\, g(\lambda) h(\lambda)^\ast d\mu(\lambda) = \int_\R \lambda\, dE_{\F^\ast g,\F^\ast h}(\lambda) \\
                                                            & = \spr{f}{\F^\ast h}_{\cH}  = \spr{\F f}{h}_{\Lmu}, \quad h\in\Lmu, 
 \end{align*}
 whenever $(\F^\ast g,f)\in\T$, which yields the claim. 
 \end{proof}

 Note that Theorem~\ref{th:TsimM} establishes a connection between the spectral properties of (the operator part of) $\T$ and $\mathrm{M}_{\mathrm{id}}$. 
 In particular, the spectrum of $\T$ coincides with the support of the measure $\mu$ and thus can be read off the singularities of $m$. 



\section{The inverse spectral problem}

 We are now ready to present the main result of this article, the solution of the inverse spectral problem for our class of generalized indefinite strings. 
 In order to state it in a concise form, we introduce the map  
\begin{align}\begin{split}
\SMP: \left\lbrace \begin{array}{rll}   \Strings      & \rightarrow & \Nevan \\
                                                     (L,\omega,\dip) &    \mapsto  & m  \end{array} \right.
\end{split}\end{align}
where $\Nevan$ is the class of all Herglotz--Nevanlinna functions and $\Strings$ is the set of generalized strings, that is, the set $\Strings$ consist of all triples $(L,\omega,\dip)$ such that $L\in(0,\infty]$, $\omega$ is a real-valued distribution in $H^{-1}_{\loc}[0,L)$ and $\dip$ is a non-negative Borel measure on $[0,L)$. 
 Our proof relies on L.\ de Branges' solution of the inverse spectral problem for canonical first order systems; see Appendix~\ref{sec:cansys} for a very brief summary.  
 More precisely, we will transform the differential equation~\eqref{eqnDEho} to a canonical system in standard form with a trace normed Hamiltonian on the semi-axis by modifying a known transformation for usual strings; cf.\ \cite{gk, kawiwo07, lawi98}.

 \begin{theorem}\label{thmIP}
  The map $\SMP$ is a bijection. 
\end{theorem}

 \begin{proof}
   {\em Injectivity.}
  With the definition in~\eqref{eqnDEThetaInt}, we consider the matrix function
  \begin{align*}
    Y(z,x) = \begin{pmatrix} \theta(z,x) & -z\phi(z,x) \\ -z^{-1}\theta^\qd(z,x) & \phi^\qd(z,x) \end{pmatrix}, \quad x\in[0,L),~ z\in\C. 
  \end{align*}
  Note that $\theta^\qd(\ledot,x)$ has a double root at zero which renders $Y$ well-defined. 
  It is an immediate consequence of~\eqref{eqnDEThetaInt} that this function satisfies the integral equation 
  \begin{align}\begin{split}\label{eqnIPYIE}
   Y(z,x) = \begin{pmatrix} 1 & 0 \\ 0 & 1 \end{pmatrix} & + z \int_0^x  \begin{pmatrix} - \Wr(t) & -1 \\ \Wr(t)^2 & \Wr(t) \end{pmatrix} Y(z,t) dt \\
                                                                                     & + z \int_{[0,x)}  \begin{pmatrix} 0 & 0 \\ 1 & 0 \end{pmatrix}Y(z,t) d\dip(t), \quad x\in[0,L),~z\in\C.
  \end{split}\end{align} 
  
  In order to transform~\eqref{eqnIPYIE} into a canonical system in standard form with a trace normed Hamiltonian, we first introduce the function $\varsigma:[0,L]\rightarrow[0,\infty]$ by 
  \begin{align}\label{eqnIPsigmatrans}
   \varsigma(x) = x + \int_0^x \Wr(t)^2 dt + \int_{[0,x)} d\dip, \quad x\in[0,L], 
  \end{align}
  as well as its generalized inverse $\xi$ on $[0,\infty)$ via 
  \begin{align}\label{eqnIPxitrans}
   \xi(s) = \sup\left\lbrace x\in[0,L)\,|\, \varsigma(x)\leq s\right\rbrace, \quad s\in[0,\infty). 
  \end{align}
  Let us point out explicitly that $\xi(s)=L$ for $s\in[\varsigma(L),\infty)$ provided that $\varsigma(L)$ is finite.
  On the other side, if $\varsigma(L)$ is not finite, then $\xi(s)<L$ for every $s\in[0,\infty)$ but $\xi(s)\rightarrow L$ as $s\rightarrow\infty$. 
  Since $\varsigma$ is strictly increasing, we infer that $\xi$ is non-decreasing and satisfies
  \begin{align}\label{eqnXiSigma}
   \xi \circ\varsigma(x) & = x, \quad x\in[0,L); &  \varsigma\circ \xi(s) & = \begin{cases} s, & s\in \ran{\varsigma},\\ \sup\,\lbrace t\in\ran{\varsigma}\,|\, t\leq s\rbrace, & s\not\in\ran{\varsigma}. \end{cases} 
  \end{align}
  Moreover, it can be deduced from~\eqref{eqnIPsigmatrans} that $\xi$ is locally absolutely continuous with
  \begin{align*}
   0\leq \xi'\leq 1
  \end{align*} 
  almost everywhere on $[0,\infty)$. 
  For future purposes, let us also mention that a substitution using, for example, \cite[Theorem~3.6.1]{bo07} shows that 
  \begin{align}\label{eqnSubstSigmaXi}
   \int_0^{\varsigma(x)} F(\xi(t))dt = \int_0^x F(t)\left(1+\Wr(t)^2\right)dt + \int_{[0,x)} F(t)d\dip(t), \quad x\in[0,L),
   \end{align}
  for every continuous function $F$ on $[0,L)$. 
  
  After these preliminary definitions, we now introduce the matrix function 
  \begin{align*}
   U(z,s)  = \begin{pmatrix} 1 & 0 \\ z(s-\varsigma\circ\xi(s)) & 1 \end{pmatrix}Y(z,\xi(s)), \quad s\in[0,\infty),~z\in\C. 
  \end{align*}
  Here note that if $\xi(s)=L$ for some $s\in[0,\infty)$, then $Y(z,\xi(s))$ has to be interpreted as the limit of $Y(z,x)$ as $x\rightarrow L$, which is known to exist in this case. 
  Since $\xi$ is locally constant on $[0,\infty)\backslash\ran{\varsigma}$,  
  we first obtain the equality 
  \begin{align}\label{eqnAuIdI}
   \xi'(s) U(z,s) = \xi'(s) Y(z,\xi(s)), \quad z\in\C, 
  \end{align}
  for almost all $s\in[0,\infty)$. 
  Moreover, the very definition of $U$ guarantees that  
  \begin{align}\label{eqnAuIdII}
\begin{pmatrix} 1 & 0 \end{pmatrix}   U(z,s)  = \begin{pmatrix} 1 & 0 \end{pmatrix} Y(z,\xi(s)), \quad s\in[0,\infty),~z\in\C.  
  \end{align}
  After using~\eqref{eqnIPYIE}, applying formula~\eqref{eqnSubstSigmaXi}, another substitution (use, for example, \cite[Corollary~5.4.4]{bo07}) as well as the identities~\eqref{eqnAuIdI} and~\eqref{eqnAuIdII}, we see that the function $U$ satisfies the integral equation 
 \begin{align}\label{eqnIEcansysSF}
  U(z,s) = \begin{pmatrix} 1 & 0 \\ 0 & 1 \end{pmatrix} - z \int_0^s  \begin{pmatrix} 0 & 1 \\ -1 & 0 \end{pmatrix} H(t)U(z,t) dt, \quad s\in[0,\infty),~z\in\C,
 \end{align}
 where  the matrix function $H$ is given by 
  \begin{align}\label{eqnHamequ}
   H(s) = \begin{pmatrix}1 - \xi'(s) & \xi'(s)\Wr(\xi(s)) \\ \xi'(s)\Wr(\xi(s)) & \xi'(s) \end{pmatrix}, \quad s\in[0,\infty).
  \end{align}
  Note that $H$ is  locally integrable (see \cite[Corollary~5.4.4]{bo07}) on $[0,\infty)$ and that $H(s)$ is real, symmetric, non-negative definite and trace normed for almost all $s\in[0,\infty)$.
  In order to be precise, non-negative definiteness holds since the function 
  \begin{align*}
   \int_0^s 1-\xi'(t) - \xi'(t)\Wr(\xi(t))^2dt & = s - \xi(s) - \int_0^{\xi(s)} \Wr(t)^2 dt \\ & = s - \varsigma\circ\xi(s) + \int_{[0,\xi(s))} d\dip, \quad s\in[0,\infty),
  \end{align*}
  is non-decreasing, 
  which shows that $\det H$ is non-negative almost everywhere. 
 
 Our next observation is that the function $m$ is also the Weyl--Titchmarsh function corresponding to the canonical system with Hamiltonian $H$ on $[0,\infty)$, that is, 
 \begin{align}\label{eqnWTmcansys}
  m(z) 
              = \lim_{s\rightarrow\infty} \frac{U_{11}(z,s)}{U_{12}(z,s)}, \quad z\in\C\backslash\R, 
 \end{align}
 which holds in view of Lemma~\ref{lemWTasQ} and the fact that $\xi(s)\rightarrow L$ as $s\rightarrow\infty$.
 As a consequence, the uniqueness part of Theorem~\ref{th:dB01} guarantees that $m$ uniquely determines $H$ almost everywhere on $[0,\infty)$. 
 In particular, this shows that the function $\xi$ is uniquely determined and hence so are $\varsigma$ and $L$. 
 Since $\xi$ maps null sets into null sets, we furthermore infer that $\Wr$ is determined almost everywhere on $[0,L)$ and thus the distribution $\omega$ is uniquely determined. 
 Finally, the fact that the measure $\dip$ is uniquely determined by $m$ as well follows from~\eqref{eqnIPsigmatrans}.  

   {\em Surjectivity.}
  Let $m\in\Nevan$ be an arbitrary Herglotz--Nevanlinna function. 
  According to the existence part of Theorem~\ref{th:dB01}, there is a locally integrable, trace normed, real, symmetric and non-negative definite $2\times2$ matrix function $H$ on $[0,\infty)$ such that the unique solution $U$ of the integral equation~\eqref{eqnIEcansysSF} satisfies~\eqref{eqnWTmcansys}.    

  We introduce the locally absolutely continuous, non-decreasing function $\xi$ by 
  \begin{align}\label{eqnIPSurDefxi}
   \xi(s) = \int_0^s H_{22}(t)dt, \quad s\in[0,\infty), 
  \end{align} 
  as well as its generalized inverse $\varsigma:[0,L]\rightarrow[0,\infty]$ via  
  \begin{align*}
   \varsigma(x) = \sup\,\lbrace s\in[0,\infty)\,|\, \xi(s)< x\rbrace\cup\lbrace 0\rbrace, \quad x\in[0,L]. 
  \end{align*}
  Here, the quantity $L\in(0,\infty]$ denotes the limit of $\xi(s)$ as $s\rightarrow\infty$ which is non-zero indeed since $H_{22}$ does not vanish almost everywhere on $[0,\infty)$. 
  The function $\varsigma$ is readily verified to be left-continuous, strictly increasing and to satisfy~\eqref{eqnXiSigma}.  
  Next we define the real-valued function $\Wr$ on $[0,L)$ by
  \begin{align*}
   \Wr(x) = \begin{cases} H_{22}(\varsigma(x))^{-1} H_{12}(\varsigma(x)), & \text{if }H_{22}(\varsigma(x))\not=0, \\ 0, & \text{if }H_{22}(\varsigma(x))=0. \end{cases}
  \end{align*}
  Since $\xi$ is constant on $[\varsigma\circ\xi(s),s]$ whenever $\varsigma\circ\xi(s)\not=s$, we infer that  
  \begin{align*}
   \xi'(s)\Wr(\xi(s)) = H_{12}(s) 
  \end{align*}
  for almost all $s\in[0,\infty)$. 
  Moreover, because $\det H$ is non-negative, we may estimate 
  \begin{align}\label{eqnWrL2}
   \xi'(s) \Wr(\xi(s))^2 \leq H_{11}(s)
  \end{align}
  for almost all $s\in[0,\infty)$. 
  In view of \cite[Corollary~5.4.4]{bo07}, this shows that $\Wr$ belongs to $L^2_{\loc}[0,L)$ and thus gives rise to a real-valued distribution $\omega$ in $H^{-1}_{\loc}[0,L)$. 
  Finally, we define the Borel measure $\dip$ on $[0,L)$ via its distribution function by 
  \begin{align*}
   \int_{[0,x)} d\dip = \varsigma(x) - x - \int_0^x \Wr(t)^2, \quad x\in[0,L). 
  \end{align*}
  In order to see that the measure $\dip$ is actually non-negative, we perform a substitution using \cite[Corollary~5.4.4]{bo07} to obtain 
  \begin{align*}
   \int_{[0,x)} d\dip =  \int_{0}^{\varsigma(x)} 1-  \xi'(s) - \xi'(s) \Wr(\xi(s))^2 ds, \quad x\in[0,L), 
  \end{align*}
  and note that the integrand on the right-hand side is non-negative almost everywhere in view of~\eqref{eqnIPSurDefxi} and~\eqref{eqnWrL2} as well as the fact that $H$ is trace normed.  
 
  With these definitions, we now introduce the matrix function $Y$ by  
  \begin{align*}
   Y(z,x) = U(z,\varsigma(x)), \quad x\in[0,L),~z\in\C. 
  \end{align*}
  Since $\xi$ is locally constant on $[0,\infty)\backslash\ran{\varsigma}$, we again obtain~\eqref{eqnAuIdI} for almost all $s\in[0,\infty)$. 
  Furthermore, the integral equation~\eqref{eqnIEcansysSF} shows that~\eqref{eqnAuIdII} holds as well. 
  Upon employing those identities, a substitution (use, for example, \cite[Corollary~5.4.4]{bo07}) and applying formula~\eqref{eqnSubstSigmaXi}, we obtain from~\eqref{eqnIEcansysSF} that the function $Y$ satisfies the integral equation~\eqref{eqnIPYIE}. 
  Clearly, the functions $Y_{11}$ and $Y_{12}$ are solutions of the homogeneous differential equation~\eqref{eqnDEho} with the initial conditions
  \begin{align*}
   Y_{11}'\NLz & = Y_{12}(z,0) = 0, & Y_{11}(z,0) & = 1, &  Y_{12}'\NLz & = -z, 
  \end{align*}
  for every $z\in\C$. 
  In view of Lemma~\ref{lemWTasQ} and~\eqref{eqnWTmcansys}, the function $m$ turns out to be the Weyl--Titchmarsh function corresponding to the triple $(L,\omega,\dip)$.
 \end{proof}
  
We conclude this section with a result about continuity properties of the correspondence $\SMP$. 
To this end, we first equip $\Nevan$ with the topology of locally uniform convergence on $\C\backslash\R$. 
Of course, the map $\SMP$ will turn out to be a homeomorphism if we endow $\Strings$ with a suitable topology. 
Instead of describing this topology in terms of bases or open sets, we simply show what convergence with respect to this topology means. 
For this purpose, let $m_n$ be a Herglotz--Nevanlinna function for every $n\in\N$ and denote all corresponding quantities as usual but with an additional subscript. 
In particular, the functions $\varsigma_n$ are defined analogously to $\varsigma$ in~\eqref{eqnIPsigmatrans}. 

\begin{proposition}\label{propSMPcont}
 The Weyl--Titchmarsh functions $m_n$ converge locally uniformly to $m$ if and only if  
 \begin{align}\label{eqnLiminfL}
    \sup\,\lbrace x\in\lbrace 0\rbrace \cup [0,\liminf_{k\rightarrow\infty} L_{n_k})\,|\, \limsup_{k\rightarrow\infty} \varsigma_{n_k}(x) < \infty \rbrace = L
  \end{align}
 holds for each subsequence $n_k$ and\footnote{Note that all quantities are well-defined for large enough $n\in\N$ in view of~\eqref{eqnLiminfL}.} 
 \begin{align}\label{eqnSMPhomeo}
   \lim_{n\rightarrow\infty} \int_0^{x} \Wr_n(t)dt & = \int_0^x \Wr(t)dt, & \lim_{n\rightarrow\infty} \int_0^x \varsigma_n(t)dt & = \int_0^x \varsigma(t) dt, 
 \end{align}
 locally uniformly for all $x\in[0,L)$. 
 Moreover, the functions $m_n$ converge locally uniformly to $\infty$ if and only if the supremum in~\eqref{eqnLiminfL} is zero for each subsequence $n_k$.
\end{proposition}

\begin{proof}
 Upon considering the corresponding Hamiltonians as in~\eqref{eqnHamequ}, we infer from Proposition~\ref{prop:dB_corresp} after a substitution (use, for example, \cite[Corollary~5.4.4]{bo07}) that the Weyl--Titchmarsh functions $m_n$ converge locally uniformly to $m$ if and only if   
 \begin{align}\label{eqnSMPcontH}
  \lim_{n\rightarrow\infty} \xi_n(s) & =\xi(s), & \lim_{n\rightarrow\infty} \int_0^{\xi_n(s)} \Wr_n(t)dt & = \int_0^{\xi(s)} \Wr(t)dt, 
 \end{align}
 locally uniformly for all $s\in[0,\infty)$. 
 Moreover, if the functions $m_n$ converge locally uniformly to $\infty$, then the functions $\xi_n$ converge locally uniformly to zero. 
 
 To begin with, we first suppose that $m_n$ converge locally uniformly to $m$. 
 Given an $x\in[0,L)$ we can find an $s\in[0,\infty)$ such that $x<\xi(s)<L$ and from~\eqref{eqnSMPcontH} we know that $x<\xi_n(s)<L$ for large enough $n\in\N$.
 By monotonicity we also have $x<\xi_n(t)\leq L_n$ for all $t\in[s,\infty)$, which shows the second inequality in 
 \begin{align}\label{eqnSMPcontLn}
  \sup\,\lbrace x\in[0,\liminf_{n\rightarrow\infty} L_{n})\,|\, \limsup_{n\rightarrow\infty} \varsigma_{n}(x) < \infty \rbrace \leq L \leq \liminf_{n\rightarrow\infty} L_n.
 \end{align}
 In order to prove the first inequality, let $x\in[0,\liminf_{n\rightarrow\infty} L_n)$ such that $\varsigma_n(x)$ is uniformly bounded for all large enough $n\in\N$.  
 For every given $\varepsilon>0$ we have 
 \begin{align*}
  |x-\xi\circ\varsigma_n(x)| = |\xi_n\circ\varsigma_n(x)-\xi\circ\varsigma_n(x)| \leq \varepsilon
 \end{align*} 
 for all large enough $n\in\N$ by~\eqref{eqnSMPcontH}. 
 Since the range of $\xi$ is bounded by $L$, this guarantees that $x-\varepsilon\leq \xi\circ\varsigma_n(x) \leq L$ and thus implies the first inequality in~\eqref{eqnSMPcontLn}. 
 Now if this inequality was strict, then we would have $\limsup_{n\rightarrow\infty} \varsigma_n(x)=\infty$ for some $x\in[0,L)$. 
 Thus for every fixed $s\in[0,\infty)$ there would be infinitely many $n\in\N$ such that 
 \begin{align*}
  \xi_n(t) \leq \xi_n\circ\varsigma_n(x) = x, \quad t\in[0,s).
 \end{align*}
 But from~\eqref{eqnSMPcontH} we would infer that $\xi$ is bounded by $x$ on $[0,s)$ as well and thus the contradiction $L\leq x$. 
 Of course, the same arguments apply to any subsequence $n_k$ which establishes~\eqref{eqnLiminfL}. 

 In order to verify the first limit in~\eqref{eqnSMPhomeo}, one notes that  
 \begin{align*}
  \frac{1}{2} \biggl|\int_0^{x} \Wr(t)-\Wr_n(t)dt\biggr|^2 &  \leq \biggl|\int_0^{\xi\circ\varsigma(x)} \Wr(t)dt - \int_0^{\xi_n\circ\varsigma(x)}\Wr_n(t)dt\biggr|^2 + \biggl| \int_{\xi_n\circ\varsigma(x)}^{\xi\circ\varsigma(x)} \Wr_n(t)dt \biggr|^2
 \end{align*}
 for every $x\in[0,L)$ and all large enough $n\in\N$ (that is, such that $x<L_n$). 
 Using the Cauchy--Schwarz inequality, the second term can be further estimated by 
 \begin{align*}
     |\xi\circ\varsigma(x)-\xi_n\circ\varsigma(x)| \, \max(\varsigma_n(x),\varsigma(x)),
 \end{align*}
 where the last factor is uniformly bounded in $n$ by~\eqref{eqnLiminfL}.
 Since $\varsigma(x)$ varies in compact subsets of $[0,\infty)$ if $x$ varies in compact subsets of $[0,L)$, we infer from~\eqref{eqnSMPcontH} that the first limit in~\eqref{eqnSMPhomeo} holds locally uniformly for all $x\in[0,L)$. 
 For the second limit in~\eqref{eqnSMPhomeo}, we first integrate by parts and use~\eqref{eqnSubstSigmaXi} to obtain  
 \begin{align*}
  \int_0^x \varsigma_n(t) - \varsigma(t)dt = \int_{\varsigma(x)}^{\varsigma_n(x)} x - \xi_n(s)ds + \int_0^{\varsigma(x)} \xi(s) - \xi_n(s)ds
 \end{align*}
 for all $x\in[0,L)$ and large enough $n\in\N$. 
 Since on the domain of integration the integrand of the first integral can be estimated by 
 \begin{align*}
  |x-\xi_n(s)| \leq |x -\xi_n\circ\varsigma(x)| = |\xi\circ\varsigma(x) - \xi_n\circ\varsigma(x)|,
 \end{align*}
 we again conclude from~\eqref{eqnSMPcontH} that the second limit in~\eqref{eqnSMPhomeo} holds locally uniformly for all $x\in[0,L)$. 
 
 Now let us suppose that the functions $m_n$ converge locally uniformly to $\infty$ and pick an arbitrary subsequence $n_k$. 
 If  $\liminf_{k\rightarrow\infty} L_{n_k}$ is zero, then clearly so is the supremum in~\eqref{eqnLiminfL}.
 Otherwise, let $x\in[0,\liminf_{k\rightarrow\infty} L_{n_k})$ such that $\varsigma_{n_k}(x)$ is uniformly bounded for all large enough $k\in\N$.
  Since the functions $\xi_n$ converge locally uniformly to zero, we then have 
 \begin{align*}
  |x| = |\xi_{n_k}\circ\varsigma_{n_k}(x)| \rightarrow 0, \qquad k\rightarrow\infty, 
 \end{align*} 
 implying that $x$ has to be zero.  Thus, the supremum in~\eqref{eqnLiminfL} is zero as well. 
 
 In order to prove the converse directions, we first assume that~\eqref{eqnLiminfL} is valid for every subsequence $n_k$ and~\eqref{eqnSMPhomeo} holds locally uniformly for all $x\in[0,L)$. 
 Let $m_{n_k}$ be a subsequence and note that by the fundamental normality test (see \cite[Section~2.7]{sc93}), this subsequence has a subsequence that converges locally uniformly either to a function $m_0$ or to $\infty$. 
 Since the supremum in~\eqref{eqnLiminfL} is non-zero, we infer that this subsequence has to converge to a function $m_0$.
 If the corresponding triple $\Xi^{-1}(m_0)$ is denoted with $(L_0,\omega_0,\dip_0)$, then the first part of the proof shows that $L_0=L$, $\omega_0=\omega$ and $\dip_0=\dip$. 
 Consequently, this guarantees that $m_0=m$ and thus we conclude that $m_n$ converges locally uniformly to $m$. 
 
 Finally, we suppose that the supremum in~\eqref{eqnLiminfL} is zero for every subsequence $n_k$. 
 As before we note that every subsequence of $m_{n}$ has a subsequence that converges locally uniformly either to a function $m_0$ or to $\infty$. 
 If this subsequence converged to a function $m_0$ with corresponding triple $(L_0,\omega_0,\dip_0)$, then the first part of the proof would imply that the supremum in~\eqref{eqnLiminfL} equals $L_0>0$ for some subsequence $n_k$. 
 Thus, we infer that the functions $m_n$ converge locally uniformly to $\infty$. 
\end{proof} 
 
 \begin{remark}\label{rem:7.3}
   Note that the locally uniform convergence of~\eqref{eqnSMPhomeo} in Proposition~\ref{propSMPcont} can be replaced by simple pointwise convergence upon employing a compactness argument. 
   In fact, an inspection of the last part of the proof of Proposition~\ref{propSMPcont} shows that locally uniform convergence is not used at all. 
 \end{remark}

\section{Non-negative strings}\label{secEssPos}

In this concluding section we will show how our solution of the inverse spectral problem fits in with the classical result by M.\ G.\ Krein \cite{kr52}. 
To this end, we will first single out all those generalized strings in $\Strings$ which give rise to purely non-negative spectrum (the non-positive case can be treated analogously).  

\begin{lemma}\label{lemWTlinterm}
 The Weyl--Titchmarsh function $m$ has the asymptotics 
 \begin{align}\label{eq:c_1=v0}
  \lim_{\eta\rightarrow\infty} \frac{m(\I\eta)}{\I\eta} = \dip(\lbrace0\rbrace). 
 \end{align}
\end{lemma}

\begin{proof}
 By Lemma~\ref{lem:c_1} and~\eqref{eqnHamequ}, we conclude that 
 \begin{align*}
  \lim_{\eta\rightarrow\infty} \frac{m(\I\eta)}{\I\eta} = \sup\,\lbrace s\in [0,\infty)\,|\, \xi(s)=0\rbrace.
 \end{align*} 
 Upon taking~\eqref{eqnIPsigmatrans} and~\eqref{eqnIPxitrans} into account, we arrive at~\eqref{eq:c_1=v0}.
 \end{proof}

 Note that Lemma~\ref{lemWTlinterm} yields the coefficient of the linear term in the integral representation~\eqref{eqnWTmIntRep} for the Weyl--Titchmarsh function $m$. 
 We will provide more detailed high energy asymptotics for $m$ as an application of the results in \cite{AsymCS}. 
 
 \begin{lemma}\label{lem:muPos}
  The spectral measure $\mu$ is supported on $[0,\infty)$ if and only if $\dip$ vanishes on $(0,L)$ and $\Wr$ has a non-decreasing representative.  
 \end{lemma} 
  
  \begin{proof}
   First of all we obtain from an integration by parts using~\eqref{eqnDEre1} and~\eqref{eqnTsaQD} that 
   \begin{align}\label{eq:qform}
     \spr{\tau f}{f}_{\cH} = -\int_0^L \Wr(x) (|\tau f_1|^2)'(x)dx + 2 \int_{[0,L)} \re(\tau f_1(x)^\ast \tau f_2(x)) d\dip(x) 
   \end{align}
  for every $\f\in\T$ such that $\tau f$ has compact support in $[0,L)$. 
  Also recall that we saw in the second part of the proof of Theorem~\ref{thmTsa} that the range of $\T$ contains all functions in $\cH$ with compact support. 
  Moreover, that proof showed that the subspace of all functions in $\cH$ with compact support in $[0,L)$ is a core for the inverse of $\T$ (which is a self-adjoint linear operator). 
  
  Under the assumption that $\mu$ is supported on $[0,\infty)$, the left-hand side of~\eqref{eq:qform} turns out to be non-negative by Theorem~\ref{th:TsimM}. 
  Now if the measure $\dip$ would not vanish on $(0,L)$, then we could find an $\f\in\T$ such that $\tau f$ has compact support in $[0,L)$ and the second integral in~\eqref{eq:qform} is non-zero. 
  Upon rescaling the second component of $\tau f$ in a suitable way, the right-hand side of~\eqref{eq:qform} would become negative, giving a contradiction. 
  Moreover, non-negativity of the right-hand side of~\eqref{eq:qform} implies that $\omega$ is a non-negative distribution on $(0,L)$. 
  In this respect one should also mention that every non-negative smooth function with compact support has a square root in $\Hasto$; see \cite[Proposition~2.1]{alkrlomi98}. 
  As a consequence, it is known \cite[Theorem~6.22]{lilo01} that $\omega$ can be represented by a non-negative Borel measure on $(0,L)$ which shows that $\Wr$ has a non-decreasing representative. 
  
   In order to prove the converse, we assume that $\dip$ vanishes on $(0,L)$ and that $\Wr$ has a non-decreasing representative.
   Since this representative is the distribution function of a non-negative Borel measure on $(0,L)$, we may integrate the first integral in~\eqref{eq:qform} by parts to see that $\spr{\tau f}{f}_{\cH}$ is non-negative for all $\f\in\T$ such that $\tau f$ has compact support in $[0,L)$.
   This readily extends to all $\f\in\T$ by continuity which guarantees that $\mu$ is supported on $[0,\infty)$ in view of Theorem~\ref{th:TsimM}. 
  \end{proof}
  
  In order to guarantee that $\omega$ arises from a non-negative Borel measure on $[0,L)$, that is, when $\Wr$ has a non-negative and non-decreasing representative, we need an additional growth restriction on the spectral measure $\mu$.     
 To this end, let us recall that a Herglotz--Nevanlinna function $m$ is called {\it a Stieltjes function} if the function $z\mapsto zm(z)$ is a Herglotz--Nevanlinna function as well \cite[Lemma S1.5.1]{kakr74a}.
 It is known that the spectral measure $\mu$ corresponding to such a function is supported on $[0,\infty)$ and satisfies a growth restriction to the extent that the integral 
 \begin{align}
  \int_{(0,\infty)} \frac{d\mu(\lambda)}{1+\lambda}
 \end{align}
 is finite. 
 In conjunction with Theorem~\ref{thmIP}, the following result immediately recovers the classical result of M.\ G.\ Krein \cite[Theorem 11.1]{kakr74} (see also \cite[Theorem 1.1]{kowa82}).
 
 \begin{proposition}\label{propWTStieltjes}
  The Weyl--Titchmarsh function $m$ is a Stieltjes function if and only if $\dip$ vanishes identically and $\omega$ is a non-negative Borel measure on $[0,L)$. 
  In this case, the function $m$ has the simplified integral representation 
  \begin{align}\label{eqnWTIntRepStieltjes}
   m(z) = \omega(\lbrace0\rbrace) - \frac{1}{Lz} + \int_{(0,\infty)} \frac{d\mu(\lambda)}{\lambda-z}, \quad z\in\C\backslash\R. 
  \end{align}
 \end{proposition}
 
\begin{proof}
 If $\dip$ vanishes identically and $\omega$ is a non-negative Borel measure on $[0,L)$, then an integration by parts shows (here we use $\psi'(z,\redot)$ to denote the unique left-continuous representative of the derivative of the solution $\psi(z,\redot)$ given by~\eqref{eq:m_pm}, which exists in this case) 
 \begin{align}\begin{split}\label{eqnWTmStieltjes}
  \frac{zm(z)-z^\ast m(z)^\ast}{z-z^\ast} &= \frac{\psi'(z,x)\psi(z,x)^* - \psi(z,x)\psi'(z,x)^*}{z-z^\ast} \\
  & \qquad\qquad +\int_{[0,x)} |\psi(z,t)|^2 d\omega(t), \quad z\in\C\backslash\R, 
 \end{split}\end{align}
 for every $x\in[0,L)$. 
Arguing as in the first part of the proof of Theorem~\ref{thmTsa}, we can show that the first term on the right-hand side of~\eqref{eqnWTmStieltjes} converges to zero along some increasing sequence $x_n\rightarrow L$. 
 In particular, this implies that $\psi(z,\redot)$ is square integrable with respect to $\omega$ and in turn that the first term on the right-hand side of~\eqref{eqnWTmStieltjes} actually converges to zero as $x$ tends to $L$. 
 Since $\omega$ is a non-negative Borel measure, this guarantees that $z\mapsto zm(z)$ is a Herglotz--Nevanlinna function as well, which means that $m$ is a Stieltjes function. 
 As such, it has an integral representation of the form~\eqref{eqnWTIntRepStieltjes} and it remains to note that 
 \begin{align}\label{eqnWTStieltjeshe}
  \lim_{\eta\rightarrow\infty} m(\I\eta) = \omega(\lbrace0\rbrace),
 \end{align}
 which holds since  the coefficient of the linear term in the integral representation of $z\mapsto zm(z)$ is given by $\omega(\lbrace0\rbrace)$ in view of \cite[\S 11.3]{kakr74} (see also \cite[Corollary~10.8]{MeasureSL}).
 
 In order to prove the converse, assume that $m$ is a Stieltjes function. 
 Then \cite[Theorem~S1.5.1]{kakr74a} implies that the coefficient of the linear term in~\eqref{eqnWTmIntRep} is zero and that the spectral measure $\mu$ is supported on $[0,\infty)$. 
 In view of Lemma~\ref{lemWTlinterm} and Lemma~\ref{lem:muPos}, this guarantees that $\dip$ vanishes identically and $\Wr$ has a non-decreasing representative.  
 Thus we are left to show that $\Wr$ is non-negative almost everywhere on $[0,L)$. 
 In order to prove this we first assume that the Weyl--Titchmarsh function $m$ is rational and show that $\Wr$ is almost everywhere equal to a step function with only a finite number of jumps. 
 To this end, let $x\in[0,L)$ be a point of increase for the non-decreasing representative of $\Wr$ and note that we have  
 \begin{align*}
  \spr{g}{\delta_x}_{\cH} = g_1(x) = 0, \quad g\in\mul{\T}, 
 \end{align*}
 since otherwise we could find an $h\in H^1_{\cc}[0,L)$ such that $\omega(g_1 h)\not=0$ which contradicts the fact that $g$ belongs to the multi-valued part of $\T$ given by (cf.\ \cite[Proposition~2.8]{bebrwe09}, \cite[Lemma~6.2]{LeftDefiniteSL}, \cite[Equation~(2.12)]{CHPencil})
 \begin{align*}
  \mul{\T} = \lbrace g\in\cH \,|\, \omega g_1 = 0 \rbrace. 
 \end{align*} 
 In particular, this shows that there can only be finitely many such points since the orthogonal complement of $\mul{\T}$ is finite dimensional by Theorem~\ref{th:TsimM}. 
 This guarantees that $\omega$ is a Borel measure on $[0,L)$, non-negative on $(0,L)$ but with a possible negative point mass at zero. 
 Upon noting that adding a real-valued constant to $\omega(\lbrace0\rbrace)$ amounts to adding the same constant to $m$, we infer from the first part of the proof that~\eqref{eqnWTStieltjeshe} holds. 
 Since $m$ is a Stieltjes function, this implies that $\omega(\lbrace0\rbrace)$ is non-negative and thus $\Wr$ is non-negative almost everywhere. 
 In the general case, we can approximate $m$ locally uniformly with rational Stieltjes functions and the claim follows from the first limit in~\eqref{eqnSMPhomeo} of Proposition~\ref{propSMPcont}. 
\end{proof}

We are also able to recover a continuity result by Y.\ Kasahara \cite{ka75} (see also \cite{ko07}) for this subclass of strings in a slightly different dressing from Proposition~\ref{propSMPcont}. 

\begin{corollary}
 Suppose that the Weyl--Titchmarsh functions $m$ and $m_n$ are Stieltjes functions for every $n\in\N$. 
 Then the functions $m_n$ converge locally uniformly to $m$ if and only if~\eqref{eqnLiminfL} holds for every subsequence $n_k$ and 
 \begin{align}\label{eq:weakW}
   \lim_{n\rightarrow\infty}  \Wr_n(x) = \Wr(x)
 \end{align}
for almost all $x\in[0,L)$.
\end{corollary}

\begin{proof}
 The claim follows from Proposition~\ref{propSMPcont} and \cite[Theorem~B]{ka12}.
\end{proof}

\appendix

\section{Canonical first order systems}\label{sec:cansys}

 The purpose of the present appendix is to briefly review some facts about canonical systems as far as they are needed in this article. 
 In particular, we will state the solution of the corresponding inverse spectral problem which is due to L.\ de Branges. 
 For more details  we refer the reader to \cite{dB60, dB61a, dB61b, dB62, dB68, ka83, lema, ro14, wi95, wi14}. 

 In order to set the stage, let $H$ be a locally integrable, real, symmetric and non-negative definite $2\times2$ matrix function on $[0,\infty)$. 
 Furthermore, we shall assume that $H$ is trace normed, that is, 
 \begin{align}\label{eq:II.03}
  \tr\, H(s) = H_{11}(s)+ H_{22}(s)=1, \quad s\in[0,\infty),  
 \end{align}
 and also exclude the cases when 
 \begin{align}
  H(s) = \begin{pmatrix} 1 & 0\\ 0 & 0 \end{pmatrix}
 \end{align}
 for almost all $s\in[0,\infty)$. 
 A matrix function $H$ with all these properties is called a Hamiltonian and associated with such a function is the canonical first order system  
 \begin{align}\label{eq:II.01}
  \begin{pmatrix} 0 & 1\\ -1 & 0 \end{pmatrix} F' = z H F,
 \end{align} 
 where $z$ is a complex spectral parameter. 
 

 Let us introduce the fundamental matrix solution $U$ of the canonical system~\eqref{eq:II.01} as the unique solution of the integral equation 
 \begin{align}
  U(z,s) = \begin{pmatrix} 1 & 0 \\ 0 & 1 \end{pmatrix} - z\int_0^s \begin{pmatrix} 0 & 1 \\ -1 & 0 \end{pmatrix} H(t) U(z,t)dt, \quad s\in[0,\infty),~z\in\C. 
 \end{align} 
 The Weyl--Titchmarsh function $m$ of the canonical system~\eqref{eq:II.01} is now defined by 
 \begin{align}
  m(z)=  \lim_{s\rightarrow\infty} \frac{U_{11}(z,s)}{U_{12}(z,s)}, \quad z\in \C\backslash\R.
 \end{align}
 As a Herglotz--Nevanlinna function, the Weyl--Titchmarsh function $m$ clearly has an integral representation \cite{kakr74a}, \cite[Section~5.3]{roro94} of the form  
 \begin{align}
  m(z) = c_1 z + c_2 +  \int_\R \frac{1}{\lambda-z} - \frac{\lambda}{1+\lambda^2} d\nu(\lambda), \quad z\in\C\backslash\R, 
 \end{align}
 for some constants $c_1$, $c_2\in\R$ with $c_1\geq0$ and a non-negative Borel measure $\nu$ on $\R$ for which the integral 
 \begin{align}
  \int_{\R}\frac{d\nu(\lambda)}{1+\lambda^2}
 \end{align}
 is finite. 
 The coefficient $c_1$ of the linear term can be read off the Hamiltonian $H$ immediately as the following fact shows; see \cite{dB61a}, \cite[Lemma~3.1]{wi95}, \cite[Lemma~2.5]{wi14}. 

\begin{lemma}\label{lem:c_1}
The Weyl--Titchmarsh function $m$ has the asymptotics  
\begin{align}
  \lim_{\eta\rightarrow\infty} \frac{m(\I\eta)}{\I\eta}=\sup\left\lbrace s\in[0,\infty) \,\left|\, H(t)=\begin{pmatrix} 1 & 0 \\ 0 & 0 \end{pmatrix} \text{ for almost all }t\in[0,x)\right.\right\rbrace. 
\end{align}
\end{lemma}
 
 It is a fundamental result of L.\ de Branges \cite{dB60, dB61a, dB61b, dB62, dB68} (see also \cite[Theorem~1]{wi95}, \cite[Theorem~2.4]{wi14}) that indeed all Herglotz--Nevanlinna functions arise as the Weyl--Titchmarsh function of a unique canonical system~\eqref{eq:II.01}.

\begin{theorem}\label{th:dB01}
 For every Herglotz--Nevanlinna function $m$ there is a Hamiltonian $H$ such that $m$ is the Weyl--Titchmarsh function of the canonical system \eqref{eq:II.01}. 
 Upon identifying Hamiltonians which coincide almost everywhere on $[0,\infty)$, this correspondence is also one-to-one.   
\end{theorem}

 The next fact establishes a continuity property for L.\ de Branges' correspondence in Theorem~\ref{th:dB01}; see \cite{dB61a}, \cite[Proposition~3.2]{lawi98}. 
 In order to state it, let $H_n$ be a Hamiltonian for every $n\in\N$ and denote with $m_n$ the corresponding Weyl--Titchmarsh function. 

\begin{proposition}\label{prop:dB_corresp}
  The Weyl--Titchmarsh functions $m_n$ converge locally uniformly to $m$ if and only if  
   \begin{align}\label{eqnHamconv}
     \lim_{n\rightarrow\infty} \int_0^x H_n(t)dt & = \int_0^x H(t)dt
   \end{align}
 locally uniformly for all $x\in[0,\infty)$. 
 Moreover, the functions $m_n$ converge locally uniformly to $\infty$ if and only if 
    \begin{align}\label{eqnHamconvinf}
     \lim_{n\rightarrow\infty} \int_0^x H_n(t)dt & =  \int_0^x \begin{pmatrix} 1 & 0 \\ 0 & 0 \end{pmatrix} dt 
   \end{align}
    locally uniformly for all $x\in[0,\infty)$. 
\end{proposition}

\begin{remark}\label{rem:dBth}
Note that the locally uniform convergence of~\eqref{eqnHamconv} and~\eqref{eqnHamconvinf} in Proposition~\ref{prop:dB_corresp} can be replaced by simple pointwise convergence upon employing a compactness argument. 
\end{remark}

\end{document}